\documentclass[a4paper,leqno,draft]{article}
\usepackage{amssymb,amsmath,amsfonts,amsthm,%eucal,
mathrsfs}
\numberwithin{equation}{section}
\newtheorem{theorem}{Theorem}[section]
\newtheorem{corollary}[theorem]{Corollary}
\newtheorem{lemma}[theorem]{Lemma}
\newtheorem{proposition}[theorem]{Proposition}
\newtheorem{definition}[theorem]{Definition}
\theoremstyle{definition}
\newtheorem{remark}[theorem]{Remark}

\newcommand{\wt}[1]{\widetilde{#1}}
\def\c{\hfill\break}

\newcommand{\Cinf}{\ensuremath{\mathcal{C}^\infty}}
\newcommand{\Cinfc}{\ensuremath{\mathcal{C}^\infty_{\text{c}}}}

\renewcommand{\S}{\mathscr{S}}

\newcommand{\mb}[1]{\ensuremath{\mathbb{#1}}}
\newcommand{\N}{\mb{N}}

\newcommand{\R}{\mb{R}}
\newcommand{\C}{\mb{C}}

\newcommand{\G}{\ensuremath{{\cal G}}}

\newcommand{\Neg}{\mathcal{N}}

\newcommand{\Ginf}{\ensuremath{\G^\infty}}

\newcommand{\lara}[1]{\langle #1 \rangle}
\newcommand{\WF}{\mathrm{WF}}
\newcommand{\WFg}{\WF_{\mathrm{g}}}
\newcommand{\singsupp}{\mathrm{sing\, supp}}
\newcommand{\supp}{\mathrm{supp}}

\newcommand{\ssc}{\mathrm{sc}}

\newfont{\bigmath}{cmr12 at 13pt}

\newfont{\grecomath}{cmmi12 at 15pt}

\newcommand{\grad}{\ensuremath{\mbox{\rm grad}\,}}
\renewcommand{\div}{\ensuremath{\mbox{\rm div}\,}}

\newcommand{\beq}{\begin{equation}}
\newcommand{\eeq}{\end{equation}}

\newcommand{\eps}{\varepsilon}

\newcommand{\Om}{\Omega}

\renewcommand{\Re}{\ensuremath{\mathrm{Re}}}

\newcommand{\mM}{\mathcal{M}}

\begin{document}

\title{{\bf Symmetrisers and generalised solutions for strictly hyperbolic systems with singular coefficients}}
\author{Claudia Garetto\footnote{Supported by JRF, Imperial College London}\\[0.1cm]
Department of Mathematics\\
Imperial College London\\
\texttt{c.garetto@imperial.ac.uk}\\
\  \\
Michael Oberguggenberger\footnote{Partially supported by FWF (Austria), grant Y237}\\[0.1cm]
Institut f\"ur Grundlagen der Bauingenieurwissenschaften\\
Leopold-Franzens-Universit\"at Innsbruck\\
\texttt{michael.oberguggenberger@uibk.ac.at}
}
%\date{ }
\maketitle
\begin{abstract}
This paper is devoted to strictly hyperbolic systems and equations with non-smooth coefficients. Below a certain level of smoothness, distributional solutions may fail to exist. We construct generalised solutions in the Colombeau algebra of generalised functions. Extending earlier results on symmetric hyperbolic systems, we introduce generalised strict hyperbolicity, construct symmetrisers, prove an appropriate G\r{a}rding inequality and establish existence, uniqueness and regularity of generalised solutions. Under additional regularity assumptions on the coefficients, when a classical solution of the Cauchy problem (or of a transmission problem in the piecewise regular case) exists, the generalised solution is shown to be associated with the classical solution (or the piecewise classical solution satisfying the appropriate transmission conditions).\\

AMS 2000 MSCS: 35S30, 46F30; 35B65.
\end{abstract}

\setcounter{section}{-1}

\section{Introduction}
This paper is devoted to generalised solutions of linear strictly hyperbolic systems and higher order equations with non-smooth coefficients.
In particular, the coefficients might be so singular as not to admit solutions in the sense of distributions. Such a situation may arise when the regularity of the coefficients is lower than Lipschitz and the data and driving terms are distributions. Formally, this leads to products of distributions. H\"older continuous coefficients or coefficients with jump discontinuities in space arise in wave propagation in non-smooth media and are of importance, e.g., in seismology \cite{HHO:08}.

In order to stay in a framework that admits solving hyperbolic equations and systems with strong coefficient singularities, we employ the Colombeau theory of algebras of generalised functions. Solvability of hyperbolic first order systems of differential and pseudodifferential equations in Colombeau algebras (or their duals) has been established in the symmetric hyperbolic case \cite{GH:03, LO:91, O:89, O:92, O:07}. What has been open is the strictly hyperbolic case and the case of higher order equations. One of the principal goals of this paper is to close this gap. For this purpose, we introduce the notion of generalised strictly hyperbolic systems, construct a generalised symmetriser and prove inequalities of G\r{a}rding type for operators with Colombeau symbols. This allows us to establish existence, uniqueness and regularity of solutions in Colombeau algebras.

A second task is to relate the generalised solutions to classical ones, when the latter exist. Indeed, in certain sufficiently regular cases, classical solutions to the Cauchy problem do exist. For example, if the coefficients are piecewise smooth with jumps across smooth hypersurfaces, one may phrase the problem classically as a transmission problem (see, e.g., \cite{AliMehmeti:94} and references therein). In this case, two classical solutions may be constructed on either side of the interface and joined by suitable jump conditions. Further, in various hyperbolic systems in divergence form (e.g., the linearised Euler system in acoustics), $L^\infty$-regularity of the coefficients and $H^1$-regularity of the data suffice for the existence of a classical solution. In such cases, we set out to show that -- as a rule -- the generalised solution is associated with the classical solution (i.e., the representing nets of smooth functions have the classical solution as their distributional limit). In the piecewise smooth case, the nets representing the Colombeau solution thus already encode a transmission condition. In physical systems (e.g., acoustics), convergence to the piecewise classical solution satisfying the physically meaningful transmission conditions will be shown. Simple results of this type have already been obtained in some special cases \cite{LO:91, O:89, O:92}. We will complete the known results in this direction by studying the linearised Euler equations in acoustics and the wave equation with discontinuous coefficients.

To set our work in perspective, we give a short survey of classical results on hyperbolic equations with non-smooth coefficients. Systems with discontinuous coefficients have already been considered in the 1950ies and 1960ies, using jump conditions and energy methods \cite{Gelfand:59, Hurd:68, Kuznecov:63}. Second order hyperbolic equations with Lipschitz coefficients that depend on time only were a subject of research starting in the 1970ies and 1980ies, culminating in the result that the Log-Lipschitz property constitutes the lower bound on regularity for having existence and uniqueness of a classical solution \cite{Colombini:79, Colombini:95}. For coefficients depending both on time and space, see also \cite{Cicognani:87}. For the general theory of pseudodifferential and paradifferential operators with symbols of low regularity we refer to the monographs \cite{Hoermander:97, Taylor:91}. Another road of investigation has been the case of transport equations with discontinuous coefficients and/or measures as initial data. Renormalised solutions have been introduced in \cite{DiPerna:89}, $BV$-vector fields have been studied in \cite{Ambrosio:04, Colombini:02}, measure theoretic concepts have been exploited in \cite{Bouchut:98, Poupaud:97}. An exploration into products of distributions is in \cite{Hoermann:01}. For an in-depth survey on these and further results we refer to \cite{Ambrosio:08, Haller:08}. In summary, this line of research has aimed at pushing the regularity to its lower bounds in special equations. That there are lower limits is shown, e.g., by an example of a first order scalar equation in divergence form with a piecewise constant coefficient that does not admit a solution in the sense of distributions, even for constant initial data \cite{Hurd:68}. In contrast to this, we place ourselves in a general framework that admits generalised solutions for arbitrary (strictly) hyperbolic equations and systems with generalised function data. The relation to the previous literature is obtained by studying the limiting behaviour of the representing nets. For example, non-uniqueness of classical solutions is reflected in the dependence on the choice of regularisation. In ongoing research \cite{Delcroix:08, GH:05, GO:10a, O:07}, we are also developing an intrinsic characterisation of the properties of the solution in case no limit exists, exploiting the asymptotic behaviour of the representing nets.

The plan of the paper is as follows: we begin by recalling some notions from the theory of Colombeau algebras as well as on Sobolev continuity of pseudodifferential operators. In Section 2,  we consider linear hyperbolic systems with Colombeau symbols, introduce the corresponding notion of strict hyperbolicity and construct a symmetriser. The symmetriser is represented by a net of smooth functions whose asymptotic bounds have to be analysed so that the right mapping properties between Colombeau algebras are obtained. Next, we show that the symmetriser can be chosen positive definite in an appropriate generalised sense. To perform the symmetrisation up to a regularising error, the asymptotic scale of the regularisation parameter has to be chosen appropriately (slow scale to have existence of generalised solutions, logarithmic slow scale to have regularity, in addition). The necessity for these types of scales has been pointed out, e.g., in \cite{GGO:03}. In Sections 3 and 4 we establish existence, uniqueness and regularity results for strictly hyperbolic systems and higher order equations, respectively, in the Colombeau setting. Coming naturally with energy estimates, the results are formulated in terms of the Colombeau algebra based on Sobolev spaces. Section 5 addresses the limiting behaviour in the case of sufficiently regular coefficients. In the system of multidimensional acoustics, we show convergence to the classical solution for $L^\infty$-coefficients and $H^1$-data. For the one-dimensional wave equation with piecewise constant coefficients, we show convergence to the distributional solution of the corresponding transmission problem. We employ Kato's perturbation result for semigroups as well as Sobolev space techniques. Finally, the Appendix contains a proof of the required G\r{a}rding inequalities for generalised pseudodifferential operators.

\section{Basic notions}
This section collects some preliminary notions concerning the different kinds of nets and of quotient spaces used in the paper. For the sake of brevity we mainly report definitions and basic properties, referring to \cite{Garetto:ISAAC07, GGO:03, GKOS:01, Ruzhansky:09} for further details.
\subsection{Colombeau theory}

\subsubsection*{Nets of numbers}

A net $(u_\eps)_\eps$ in $\C^{(0,1]}$ is said to be \emph{strictly nonzero} if there exist $r>0$ and $\eta\in(0,1]$ such that $|u_\eps|\ge \eps^r$ for all $\eps\in(0,\eta]$.

For several regularity issues we will make use of the following concepts of \emph{slow scale net} and \emph{logarithmic slow scale net}. A net $\omega_\eps\in\R^{0,1]}$ is a \emph{slow scale net} iff $|\omega_\eps|=O(\eps^{-p})$ for all $p>0$. We say that $\omega_\eps\in\R^{0,1]}$ is a \emph{logarithmic slow scale net} if $|\omega_\eps|=O(\log^p(1/\eps))$ for all $p>0$. Note that the suffices \emph{sc} and \emph{lsc} will stand for slow scale and logarithmic slow scale respectively.

\subsubsection*{Generalised functions based on a locally convex topological vector space $E$}
The most common algebras of generalised functions of Colombeau type as well as the spaces of generalised symbols we deal with are introduced by means of the following general models.

Let $E$ be a locally convex topological vector space topologised through the family of seminorms $\{p_i\}_{i\in I}$. The elements of
\[
\begin{split}
\mM_E &:= \{(u_\eps)_\eps\in E^{(0,1]}:\, \forall i\in I\,\, \exists N\in\N\quad p_i(u_\eps)=O(\eps^{-N})\, \text{as}\, \eps\to 0\},\\
\mM^\ssc_E &:=\{(u_\eps)_\eps\in E^{(0,1]}:\, \forall i\in I\,\, \exists (\omega_\eps)_\eps\, \text{s.s.n.}\quad p_i(u_\eps)=O(\omega_\eps)\, \text{as}\, \eps\to 0\},\\
\mM^\infty_E &:=\{(u_\eps)_\eps\in E^{(0,1]}:\, \exists N\in\N\,\, \forall i\in I\quad p_i(u_\eps)=O(\eps^{-N})\, \text{as}\, \eps\to 0\},\\
\Neg_E &:= \{(u_\eps)_\eps\in E^{(0,1]}:\, \forall i\in I\,\, \forall q\in\N\quad p_i(u_\eps)=O(\eps^{q})\, \text{as}\, \eps\to 0\},
\end{split}
\]

are called $E$-moderate, $E$-moderate of slow scale type, $E$-regular and $E$-negligible, respectively. We define the space of \emph{generalised functions based on $E$} as the factor space $\G_E := \mM_E / \Neg_E$.

The ring of \emph{complex generalised numbers}, denoted by $\wt{\C}$, is obtained by taking $E=\C$. We remark that $\wt{\C}$ is not a field since by Theorem 1.2.38 in \cite{GKOS:01} only the elements which are strictly nonzero (i.e. the elements which have a representative strictly nonzero) are invertible and vice versa. Note that all the representatives of $u\in\wt{\C}$ are strictly nonzero as soon as there exists one representative which is strictly nonzero.
%When $u$ has a representative which is slow scale-strictly nonzero we say that it is \emph{slow scale-invertible}.

For any locally convex topological vector space $E$ the space $\G_E$ has the structure of a $\wt{\C}$-module. The ${\C}$-module $\G^\ssc_E:=\mM^\ssc_E/\Neg_E$ of \emph{slow scale regular generalised functions} and the $\wt{\C}$-module $\Ginf_E:=\mM^\infty_E/\Neg_E$ of \emph{regular generalised functions} are subrings of $\G_E$, which are characterised by more specific moderateness properties at the level of representatives. We use the notation $u=[(u_\eps)_\eps]$ for the class $u$ of $(u_\eps)_\eps$ in $\G_E$. This is the usual way adopted in the paper to denote an equivalence class.

\subsubsection*{Colombeau algebras and spaces of generalised symbols used in the paper}
The Colombeau algebra $\G(\Om)$ on an $\Om$ open subset of $\R^n$ is the factor space $\G_E$ with $E=\Cinf(\Om)$. For the hyperbolic problems studied in this paper we will mainly choose $E=\Cinf_{\rm{b}}([-T,T]\times\R^n)$, the space of infinitely differentiable functions all whose derivatives are bounded, and work in the corresponding Colombeau algebra $\G_{\rm{b}}([-T,T]\times\R^n)=\G_{\Cinf_{\rm{b}}([-T,T]\times\R^n)}$. It will be also convenient to work with Colombeau generalised functions based on Sobolev spaces. In particular, $\G_{2,2}(\R^n):=\G_{H^\infty(\R^n)}$ and $\G_{2,2}((-T,T)\times\R^n):=\G_{H^\infty((-T,T)\times\R^n)}$.

With the expression \emph{generalised symbol} we mean an element of the Colombeau space $\G_E$, where $E$ is a suitable space of symbols. For example $E=S^m(\R^{2n})$, $E=\Cinf([-T,T], S^m(\R^{2n}))$, etc. Specific moderateness properties can be required at the level of representatives. We recall that the \emph{slow scale regular generalised symbols} ($a=[(a_\eps)_\eps]\in\G^\ssc_{S^m(\R^{2n})}$) are those generated by nets in $\mM^\ssc_{S^m(\R^{2n})}$: more precisely,
\[
(a_\eps)_\eps\in \mM^\ssc_{S^m(\R^{2n})}\quad \Longleftrightarrow\quad \text{for all }\alpha,\beta\in\N^n\ \text{there exists a slow scale net }\omega_\eps\
\text{such that } |a_\eps|_{\alpha,\beta}^{(m)}=O(\omega_\eps),
\]
where $|a|^{(m)}_l$ denote the seminorm $\sup_{(x,\xi)\in\R^{2n}, |\alpha+\beta|\le l}\lara{\xi}^{-m+|\alpha|}|\partial^\alpha_\xi\partial^\beta_x a(x,\xi)|$.
Analogously one can define symbols which are \emph{logarithmic slow scale regular} by taking representing nets with the following property:
\[
\text{for all }\alpha,\beta\in\N^n\ \text{there exists a logarithmic slow scale net }\omega_\eps\ \text{such that } |a_\eps|_{\alpha,\beta}^{(m)}=O(\omega_\eps).
\]
We use the notations $(a_\eps)_\eps\in\mM^{\rm{lsc}}_{S^m(\R^{2n})}$ and $a\in\G^{\rm{lsc}}_{S^m(\R^{2n})}$.
In \cite{Garetto:ISAAC07, GGO:03} a complete symbolic calculus has been developed for slow scale regular generalised symbols. In the same vein, a complete symbolic calculus can be developed for logarithmic slow scale regular generalised symbols as well.

In general, the net of representing smooth functions $(u_\varepsilon)_\varepsilon$ of an element $u \in  \G(\Om)$ does not converge. However, if it converges weakly to a distribution $v\in {\mathcal D}'(\Om)$, then $u$ is said to admit $v$ as \emph{associated distribution}.

\subsection{Sobolev-boundedness of generalised pseudodifferential operators}
We collect some basic facts concerning pseudodifferential operators which will be employed in the sequel. Our main reference is \cite{Kumano-go:81}.
\subsubsection*{Sobolev spaces}
We recall that $\lara{D_x}^s$ denotes the pseudodifferential operator with symbol $\lara{\xi}^s=(1+|\xi|^2)^{\frac{s}{2}}$ and that the Sobolev space $H^s(\R^n)$ is the set of all $u\in\S'(\R^n)$ such that $\lara{D_x}^s u\in L^2(\R^n)$. $H^s(\R^n)$ is a Banach space with respect to the norm $\Vert u\Vert_s=\Vert \lara{D_x}^s u\Vert_{L^2(\R^n)}$. When $u\in H^s(\R^n)$ and $v\in H^{-s}(\R^n)$, the inequality $|\langle u,v\rangle)|\le \Vert u\Vert_s \Vert v\Vert_{-s}$ holds for the distributional action of $v$ on $u$.
\subsubsection*{Sobolev-boundedness theorem for pseudodifferential operators on $\R^n$}
Let $S^m(\R^{2n})$ be the space of H\"ormander symbols of order $m$ satisfying global estimates in $x$ and $\xi$. The corresponding pseudodifferential operator $a(x,D)$ defines a continuous map from $H^{s+m}(\R^n)$ to $H^s(\R^n)$. More precisely, see \cite[Theorem 2.7]{Kumano-go:81}, for each $s$ there exist a constant $C_s$ and an integer $l_s$ such that
\beq
\label{sob_continuity}
\Vert a(x,D)u\Vert_s\le C_s|a|^{(m)}_{l_s}\Vert u\Vert_{s+m}
\eeq
for all $u\in H^{s+m}(\R^n)$. Note that the constants $C_s$ and $l_s$ only depend on $s$ and the order of the symbol $a$.
\subsubsection*{Sobolev-boundedness of nets of pseudodifferential operators and of generalised pseudodifferential operators}
Let us now consider a net of symbols $(a_\eps)_\eps\in S^m(\R^{2n})^{(0,1]}$. It follows from the above that the inequality \eqref{sob_continuity} holds uniformly with respect to $\eps$: in other words,
\beq
\label{sob_continuity_eps}
\Vert a_\eps(x,D)u\Vert_s\le C_s|a_\eps|^{(m)}_{l_s}\Vert u\Vert_{s+m}
\eeq
for all $\eps\in(0,1]$ and $u\in H^{s+m}(\R^n)$. In particular, the net $(C_s|a_\eps|^{(m)}_{l_s})_\eps$ of continuity constants will inherit the moderateness property of the nets of symbols $(a_\eps)_\eps$. Hence, a generalised pseudodifferential operator $a(x,D)$ with symbol $a\in\G_{S^m(\R^{2n})}$ is well-defined as a map
\[
a(x,D):\G_{H^{s+m}(\R^n)}\to \G_{H^s(\R^n)}.
\]
This map is continuous with respect to the corresponding sharp topologies on the Colombeau spaces $\G_{H^{s+m}(\R^n)}$ and $\G_{H^s(\R^n)}$ (see \cite{Garetto:05a} for the definition of sharp topology).
\begin{remark}
\label{rem_G_2,2}
Let $\G_{2,2}(\R^n)$ be the Colombeau space based on $H^\infty(\R^n)$. From \eqref{sob_continuity_eps} it immediately follows that any generalised pseudodifferential operator $a(x,D)$ maps $\G_{2,2}(\R^n)$ continuously into itself.
\end{remark}

\section{Symmetrisable systems}
This section is devoted to first order systems of the following type:
\beq
\label{system}
\frac{\partial}{\partial t}u=Ku+f,
\eeq
where $K(t,x,D_x)$ is an $m\times m$-matrix of generalised pseudodifferential operators $k_{i,j}(t,x,D_x)$ with symbol in the space $\G_{\Cinf([-T,T], S^1(\R^{2n}))}$. We use the notation $K(t,x,\xi)$ for the corresponding matrix of symbols with entries $k_{i,j}(t,x,\xi)$, $i,j=1,...,m$. We say that $a\in\G_{\Cinf([-T,T], S^m(\R^{2n}))}$ is a \emph{classical} (or \emph{polyhomogeneous}) generalised symbol if there exists a sequence of symbols $a_{m-j}\in\G_{\Cinf([-T,T], S^{m-j}(\R^{2n}))}$, a sequence of representatives $(a_{m-j,\eps})_\eps$ and a representative $(a_\eps)_\eps$ of $a$ such that the following holds:
\begin{itemize}
\item[-] $a_{m-j,\eps}$ is homogeneous of degree $m-j$ in $\xi$ for $|\xi|\ge 1$, i.e.,
\[
a_{m-j,\eps}(t,x,s\xi)=s^{m-j}a_{m-j,\eps}(t,x,\xi)
\]
for $|\xi|\ge 1$ and $r\ge 1$;
\item[-] for each $r$ the net $(a_\eps-\sum_{j=0}^r a_{m-j,\eps})_\eps\in\mM_{\Cinf([-T,T], S^{m-r-1}(\R^{2n}))}$.
\end{itemize}
We refer to $a_m\in\G_{\Cinf([-T,T], S^m(\R^{2n}))}$ as the \emph{principal part} of $a\in\G_{\Cinf([-T,T], S^m(\R^{2n}))}$ or \emph{principal symbol} of the operator $a(x,D)$. Clearly it is uniquely defined modulo $\G_{\Cinf([-T,T], S^{m-1}(\R^{2n}))}$. Note that when one works with more specific kinds of moderateness (i.e., slow scale, logarithmic slow scale) the same kind of moderateness is involved in the corresponding notion of asymptotic expansion. Hence the principal part of a slow scale (or logarithmic slow scale) regular classical symbol is slow scale (or logarithmic slow scale) regular as well.

\subsection{Symmetriser: definition and existence in the generalised strictly hyperbolic case}
From now on we assume that $K(t,x,\xi)$ has a classical generalised symbol, i.e., with entries in the space $\G_{\Cinf([-T,T], S^1(\R^{2n}))}$ and we denote its principal symbol by $K_1(t,x,\xi)$. We introduce the following notion of symmetriser and symmetrisable system.
\begin{definition}
\label{def_symmetriser}
A symmetriser for $\partial/\partial t-K$ is a smooth one parameter family of operators $R=R(t,x,D_x)$ such that
\begin{itemize}
\item[(i)] the entries of $R$ are given by classical generalised symbols in $\G_{\Cinf([-T,T], S^0(\R^{2n}))}$;
\item[(ii)] $R_0(t,x,\xi)$ is a positive definite matrix for $|\xi|\ge 1$, in the sense that there exists a representative $(R_{0,\eps})_\eps$ and a constant $c>0$ such that
\[
\overline{z}^{T}R_{0,\eps}(t,x,\xi)z\ge c|z|^2
\]
for all $z\in\C^m$, for all $t\in[-T,T]$, $x\in\R^n$, $\eps\in(0,1]$ and for all $|\xi|\ge 1$
\item[(iii)] $RK+(RK)^\ast$ is a matrix of pseudodifferential operators with symbol in $\G_{\Cinf([-T,T], S^0(\R^{2n}))}$.
\end{itemize}
If such a symmetriser exists, we say that $\partial/\partial t-K$ is symmetrisable.
\end{definition}
\begin{definition}
\label{def_sh}
The operator $\partial/\partial t-K$ is called \emph{generalised strictly hyperbolic} if there exists a representative $(K_{1,\eps})_\eps$ of the principal symbol $K_1(t,x,\xi)$ having $m$ pure imaginary eigenvalues $(i\lambda_{j,\eps})_\eps$ and a strictly nonzero net $(\lambda_\eps)_\eps$ such that the bound from below
\beq
\label{bound_fb}
\lambda_{j+1,\eps}(t,x,\xi)-\lambda_{j,\eps}(t,x,\xi)\ge \lambda_\eps\lara{\xi}
\eeq
holds for all $j=1,...,m-1$, for all $\eps\in(0,1]$, $t\in[-T,T]$, $x\in\R^n$ and for $\xi$ away from zero.
\end{definition}
The expression \emph{$\xi$ away from zero} stands for $|\xi|\ge r$ with $0<r<1$.

After truncation at zero, the eigenvalues $i\lambda_j$ in Definition \ref{def_sh} generate symbols with the same asymptotic properties (in $\varepsilon$) as the matrix $K_1(t,x,\xi)$. Indeed, using arguments as in the proof of Proposition 6.4 in \cite{Mizohata:73} one gets the following result.
\begin{lemma}
\label{lemma_eigen}
The nets $(\lambda_{j,\eps})$ generates $m$ distinct and real classical symbols $\lambda'_j$ in $\G_{\Cinf([-T,T], S^1(\R^{2n}))}$ such that $\lambda'_j(t,x,\xi)=\lambda_j(t,x,\xi)$ for $|\xi|$ large enough.
\end{lemma}
\begin{proof}
We work at the level of representatives. Let $P_\eps(\lambda;t,x,\xi):=({\rm{det}}(\lambda I+iK_{1,\eps}(t,x,\xi))=\Pi_{1\le j\le m}(\lambda-\lambda_{j,\eps}(t,x,\xi))$. The bound from below \eqref{bound_fb} yields
\[
\biggl|\biggl(\frac{\partial P_\eps}{\partial\lambda}\biggr)_{\lambda=\lambda_{j,\eps}}\biggr|\ge c'_\eps\lara{\xi}^{m-1}
\]
for $t\in[-T,T]$, $x\in\R^n$ and $|\xi|\ge r$. Since each eigenvalue can be bounded by the norm of the matrix $K_{1,\eps}(t,x,\xi)$ we have the estimate
\[
|\lambda_{j,\eps}(t,x,\xi)|\le \mu_\eps\lara{\xi},
\]
valid for $t\in[-T,T]$ and $(x,\xi)\in\R^{2n}$, where $(\mu_\eps)_\eps$ is a moderate net depending on the coefficients of $K_{1,\eps}$. Passing to the first order derivatives from $P_\eps(\lambda;t,x,\xi)=\Pi_{1\le j\le m}(\lambda-\lambda_{j,\eps}(t,x,\xi))$ we get, for $|\xi|\ge r$,
\beq
\label{1_deriv_Mizo}
\frac{\partial\lambda_{j,\eps}}{\partial x_k}=-\biggl(\frac{\frac{\partial P_\eps}{\partial x_k}}{\frac{\partial P_\eps}{\partial\lambda}}\biggr)_{\lambda=\lambda_{j,\eps}},\qquad \frac{\partial\lambda_{j,\eps}}{\partial \xi_k}=-\biggl(\frac{\frac{\partial P_\eps}{\partial \xi_k}}{\frac{\partial P_\eps}{\partial\lambda}}\biggr)_{\lambda=\lambda_{j,\eps}}.
\eeq
An induction argument shows that each eigenvalue $\lambda_{j,\eps}(t,x,\xi)$ satisfies moderate symbol estimates of order $1$ for $|\xi|\ge r$. From the fact that $P_\eps$ is homogeneous of order $m$ in $\xi$ for $|\xi|\ge r$ we obtain that each net $(\lambda_{j,\eps})_\eps$ is homogeneous of order $1$ on the same $\xi$-domain. Let now $\psi\in\Cinf(\R^n)$, $\psi(\xi)=0$ for $|\xi|\le r$ and $\psi(\xi)=1$ for $|\xi|\ge \alpha r$, with $\alpha r<1$. The net $\lambda'_{j,\eps}(t,x,\xi)=\psi(\xi)\lambda_{j,\eps}(t,x,\xi)$ generates a real classical symbol $\lambda'_j$ in $\G_{\Cinf([-T,T], S^1(\R^{2n}))}$ which coincides with $\lambda_j$ for $|\xi|\ge \alpha r$.
\end{proof}
\begin{remark}
From the previous lemma it is clear that when $K_1$ has slow scale entries then the eigenvalues are slow scale regular as well. Note that one can slightly generalise Definition \ref{def_sh} by allowing a strictly nonzero radius $(r_\eps)_\eps$. Finally, in order to get slow scale nets of eigenvalues one combines the slow scale assumption on the entries of $K_1$ with the condition that $(r_\eps)_\eps$ as well as $(\lambda_\eps)_\eps$ is the inverse of a slow scale net.
\end{remark}
\begin{remark}
\label{remark_adjoint}
In the sequel we will make use of the adjoint $A(t,x,D_x)^\ast$ of a matrix $A(t,x,D_x)$ of pseudodifferential operators. By definition $A(t,x,D_x)^\ast$ fulfills the relation $(A(t,x,D_x)u,v)_{{L^2(\R^n)}^m}=(u,A(t,x,D_x)^\ast v)_{{L^2(\R^n)}^m}$ for all $u,v\in\S(\R^n)$. At the symbol level we have that, modulo lower order terms, $A(t,x,\xi)^\ast=\overline{A(t,x,\xi)}^T$. All these statements are valid for generalised pseudodifferential operators as well.
\end{remark}

\begin{proposition}
\label{prop_symmetriser}
Any generalised strictly hyperbolic operator $\partial/\partial t-K$ is symmetrisable.
\end{proposition}
\begin{proof}
We have to construct an operator matrix $R$ fulfilling Definition \ref{def_symmetriser}. We begin by writing the projections $P_{j,\eps}$ onto the associated eigenspaces of $i\lambda_{j,\eps}$. For the moment we take $\xi\neq 0$ and we work with $\xi/|\xi|$. Since from the hypothesis of generalised strict hyperbolicity we have
\[
\lambda_{j+1,\eps}(t,x,\xi/|\xi|)-\lambda_{j,\eps}(t,x,\xi/|\xi|)\ge \lambda_\eps,
\]
with a strictly nonzero net $(\lambda_\eps)_\eps$, we can find a curve $\gamma_{j,\eps}(t,x,\xi/|\xi|)$ around the point $i\lambda_{j,\eps}(t,x,\xi/|\xi|)$ such that all $i\lambda_{h,\eps}(t,x,\xi/|\xi|)$ with $ h\neq j$ do not belong to the closed domain with boundary $\gamma_{j,\eps}(t,x,\xi/|\xi|)$. Hence,
\[
P_{j,\eps}(t,x,\xi/|\xi|)=\frac{1}{2\pi i}\int_{\gamma_{j,\eps}}(zI-K_{1,\eps}(t,x,\xi/|\xi|))^{-1}\, dz.
\]
Easy computations, as in \cite[Proposition 2]{RuzhaWirth:08} lead to the following formula
\beq
\label{formula_Pj}
P_{j,\eps}(t,x,\xi/|\xi|)=\prod_{h\neq j}\frac{-iK_{1,\eps}(t,x,\xi/|\xi|)-\lambda_{h,\eps}(t,x,\xi/|\xi|)I}{\lambda_{j,\eps}(t,x,\xi/|\xi|)-\lambda_{h,\eps}(t,x,\xi/|\xi|)}
\eeq
The hypothesis \eqref{bound_fb} of strict hyperbolicity combined with the fact that $K_{1,\eps}$ is homogeneous of order $1$ in $\xi$ implies that $P_{j,\eps}$ is homogeneous of order $0$ in $\xi$ (for $\xi\neq 0$). More precisely, from Lemma \ref{lemma_eigen} and \eqref{formula_Pj} we easily obtain that $P_{j,\eps}(t,x,\xi)=P_{j,\eps}(t,x,\xi/|\xi|)$ is a matrix with entries in $\mM_{\Cinf([-T,T], S^0(\R^{2n}\setminus 0))}$. Truncating around $0$ (for instance taking a cut-off function identically $0$ for $|\xi|\le 1/2$ and identically $1$ for $|\xi|\ge 1$) and keeping the same operator notation we obtain that the operators $P_j(t,x,D_x)=[(P_{j,\eps}(t,x,D_x))_\eps]$ are defined by matrices $P_j(t,x,\xi)$ of classical generalised symbols in $\G_{\Cinf([-T,T], S^0(\R^{2n}))}$. Let now
\[
R(t,x,D_x):=\sum_{j=1}^m P_j(t,x,D_x)^\ast P_j(t,x,D_x),
\]
where for the definition of $ P_j(t,x,D_x)^\ast$ we refer to Remark \ref{remark_adjoint}. The entries of the corresponding matrix $R(t,x,\xi)$ are classical generalised symbols in $\G_{\Cinf([-T,T], S^0(\R^{2n}))}$. Moreover, the principal symbol $R_0(t,x,\xi)$ is given by $\sum_{j=1}^m\overline{P_j(t,x,\xi)}^{\,T} P_j(t,x,\xi)$. Choosing the representatives $P_{j,\eps}$ as above we have that $R_{0,\eps}(t,x,\xi)=\sum_{j=1}^m P_{j,\eps}(t,x,\xi)^\ast P_{j,\eps}(t,x,\xi)$ fulfills the following bound from below for any $z\in\C^n$ and for $|\xi|\ge 1$:
\[
\overline{z}^{T}R_{0,\eps}(t,x,\xi)z=\sum_{j=1}^m\vert P_{j,\eps}(t,x,\xi)z\vert^2\ge \frac{1}{m^2}\vert z\vert^2.
\]
Before proceeding we observe for $\xi\neq 0$ we have $\sum_{j=1}^m i\lambda_j(t,x,\xi)P_j(t,x,\xi)=K_1(t,x,\xi)$, ${P_j}(t,x,\xi)^2=P_j(t,x,\xi)$ and $P_jP_h(t,x,\xi)=0$ for $j\neq h$. It follows that modulo $\G_{\Cinf([-T,T], S^0(\R^{2n}))}$ the symbol of $RK$ is given by
\[
i\sum_{j=1}^m\lambda_j(t,x,\xi)\overline{P_j(t,x,\xi)}^T P_j(t,x,\xi).
\]
Hence, once more modulo $\G_{\Cinf([-T,T], S^0(\R^{2n}))}$ the symbol of $(RK)^\ast$ is
\[
-i\sum_{j=1}^m\lambda_j(t,x,\xi)\overline{P_j(t,x,\xi)}^T P_j(t,x,\xi).
\]
and then $(RK)+(RK)^\ast$ is a matrix of generalised pseudodifferential operators of order $0$. This means that $R$ is a symmetriser for the operator $\partial/\partial t-K$.

\end{proof}
%\begin{remark}
%\label{rem_sym}
\subsection{An $L^2$-positive definite symmetriser and the corresponding $L^2$-energy estimates}
Note that the symmetriser $R$ constructed above is $L^2$-self-adjoint, i.e., $R=R^\ast$. Moreover a straightforward application of the G\r{a}rding inequality as stated in the Appendix yields a symmetriser $S$ which is positive definite on $L^2(\R^n)$. In the sequel $(\cdot,\cdot)$ denotes the scalar product in $L^2$. The corresponding norm is denoted by $\Vert\cdot\Vert$ or equivalently $\Vert\cdot\Vert_0$ (in the Sobolev space notation).
\begin{proposition}
\label{prop_symm_S}
Let $R$ be the symmetriser constructed in the proof of Proposition \ref{prop_symmetriser}. Then there exists a symmetriser $S$ such that
\begin{itemize}
\item[(i)] $R=S$ (modulo a matrix of pseudodifferential operators with symbol in $\G_{\Cinf([-T,T], S^{-1}(\R^{2n}))}$),
\item[(ii)] $S$ is $L^2$-positive definite, i.e. there exists a representative $(S_{\eps})_\eps$ and a constant $c>0$ such that
\beq
\label{S_pos}
(S_{\eps}u,u)\ge c\Vert u\Vert^2
\eeq
for all $u\in {L^2(\R^n)}^m$ and all $\eps\in(0,1]$.
\end{itemize}
\end{proposition}
%and
%\beq
%\label{est_Ru}
%(Ru,u)_{{L^2(\R^n)}^m}=\biggl(\sum_{j=1}^m P_j(t,x,D_x)^\ast P_j(t,x,D_x)u,u\biggr)_{{L^2(\R^n)}^m}=\sum_{j=1}^m \Vert P_j(t,x,D_x)u\Vert^2_{{L^2(\R^n)}^m},
%\eeq
%where \eqref{est_Ru} is valid for all $u\in ({\G_{L^2(\R^n)}})^m$.
%\end{remark}
\begin{proof}
We begin by estimating $\Re\, R_\eps(t,x,\xi)$. Since
\[
\Re\, R_\eps(t,x,\xi)=R_{0,\eps}(t,x,\xi)+\frac{R_{-1,\eps}(t,x,\xi)+R_{-1,\eps}(t,x,\xi)^\ast}{2},
\]
we obtain from the properties of $R_{0,\eps}$ and the symbol order of ${R_{-1,\eps}+R^\ast_{-1,\eps}}$ that
\[
\overline{z}^T\Re\, R_\eps(t,x,\xi)z\ge \frac{1}{m^2}|z|^2-\lara{\xi}^{-1}\mu_\eps|z|^2\ge \frac{1}{m^2}|z|^2-|\xi|^{-1}\mu_\eps|z|^2.
\]
This bound from below holds for $t\in[-T,T]$, $(x,\xi)\in\R^{2n}$ and $\eps\in(0,1]$ with the moderate net $(\mu_\eps)_\eps$ depending on the symbol estimates of ${R_{-1,\eps}+R^\ast_{-1,\eps}}$. Hence for $r_\eps:=2m^2\mu_\eps$ and for $|\xi|\ge r_\eps$ we get
\[
\overline{z}^T\Re\, R_\eps(t,x,\xi)z\ge \frac{1}{2m^2}|z|^2.
\]
It follows that for $c=1/(2m^2)$ the inequality
\[
\Re\,R_\eps(t,x,\xi)\ge cI
\]
holds for $|\xi|\ge r_\eps$. An application of the G\r{a}rding inequality in Theorem \ref{theo_Garding} entails
\[
\Re(R_\eps(t,x,D_x)u,u)\ge c\Vert u\Vert^2_0-c_{1,\eps}\Vert u\Vert^2_{-1/2},\qquad\qquad u\in L^2(\R^n),
\]
for some moderate net $(c_{1,\eps})_\eps$. Let us now define
\[
S_\eps(t,x,D_x)=R_\eps(t,x,D_x)+c_{1,\eps}\lara{D_x}^{-1}I.
\]
It is clear that the corresponding generalised operator $S$ is equal to $R$ modulo a matrix of operators of order $-1$. Furthermore it is $L^2$-positive definite. Indeed, since $S_\eps(t,x,D_x)^\ast=S_\eps(t,x,D_x)$ we have that $(S_{\eps}u,u)$ is real and from the G\r{a}rding inequality above
\begin{multline*}
(S_{\eps}u,u)= (R_\eps u,u)+c_{1,\eps}(\lara{D_x}^{-1}Iu,u)\ge c\Vert u\Vert^2_0-c_{1,\eps}\Vert u\Vert^2_{-1/2}+c_{1,\eps}(\lara{D_x}^{-1}Iu,u)\\
=c\Vert u\Vert^2_0-c_{1,\eps}\Vert u\Vert^2_{-1/2}+c_{1,\eps}(\lara{D_x}^{-1/2}Iu,\lara{D_x}^{-1/2}Iu)\\
=c\Vert u\Vert^2_0-c_{1,\eps}\Vert u\Vert^2_{-1/2}+c_{1,\eps}\Vert u\Vert^2_{-1/2}.
\end{multline*}
Concluding
\[
(S_{\eps}u,u) \ge c\Vert u\Vert^2_0.
\]
\end{proof}

The symmetriser $S$ yields an immediate energy estimate for the net of solutions $(u_\eps)_\eps$ to the equation
\beq
\label{eq_repr}
\frac{\partial}{\partial t}u_\eps=K_\eps(t,x,D_x)u_\eps+f_\eps,
\eeq
with initial data $u_\eps(0,\cdot)=g_\eps$. More precisely, we assume that the operator $\partial/\partial t-K$ is generalised strictly hyperbolic, we take $(K_\eps)_\eps$ with principal part $(K_{1,\eps})_\eps$ as in Definition \ref{def_sh}, $f_\eps\in \Cinf([-T,T], L^2(\R^n))^m$ and $g_\eps\in {L^2(\R^n)}^m$. We have
\beq
\label{ee_0}
\begin{split}
\frac{d}{dt}(S_\eps u_\eps,u_\eps)&=(\partial_t S_\eps u_\eps,u_\eps)+(S_\eps\partial_t u_\eps,u_\eps)+(S_\eps u_\eps,\partial_t u_\eps)\\
&=(\partial_t S_\eps u_\eps,u_\eps)+(S_\eps K_\eps u_\eps +S_\eps f_\eps,u_\eps)+(S_\eps u_\eps,K_\eps u_\eps +f_\eps)\\
&=(\partial_t S_\eps u_\eps,u_\eps)+ ((S_\eps K_\eps+K_\eps^\ast S_\eps)u_\eps,u_\eps)+(S_\eps f_\eps, u_\eps)+(S_\eps u_\eps,f_\eps)\\
&\le c_\eps(\partial_t S_\eps)\Vert u_\eps\Vert^2+c_\eps(S_\eps K_\eps+K_\eps^\ast S_\eps)\Vert u_\eps\Vert^2+2c_\eps(S_\eps)\Vert f_\eps\Vert \Vert u_\eps\Vert\\
&\le c_\eps(\partial_t S_\eps)\Vert u_\eps\Vert^2+c_\eps(S_\eps K_\eps+K_\eps^\ast S_\eps)\Vert u_\eps\Vert^2+c_\eps(S_\eps)(\Vert u_\eps\Vert^2+\Vert f_\eps\Vert^2)\\
&\le C_\eps(\partial_t S_\eps,S_\eps,S_\eps K_\eps+K_\eps^\ast S_\eps)(\Vert u_\eps\Vert^2+\Vert f_\eps\Vert^2)\\
&\le C_\eps c^{-1}(S_\eps u_\eps,u_\eps)+C_\eps\Vert f_\eps\Vert^2,
\end{split}
\eeq
where the net $(C_\eps)_\eps$ depends on symbol seminorms of $(S_\eps)_\eps$, $(\partial_t S_\eps)_\eps$ and  $(S_\eps K_\eps+K_\eps^\ast S_\eps)_\eps$ and the constant $c$ appears in the bound from below \eqref{S_pos}.

An application of Gronwall's lemma gives the estimate
\beq
\label{1_ineq}
\Vert u_\eps(t,\cdot)\Vert^2\le c^{-1}(S_\eps u_\eps,u_\eps)\le c^{-1}d_\eps \exp({c^{-1}C_\eps t}),
\eeq
valid for all $\eps\in(0,1]$ and $t\in[-T,T]$ with
\[
d_\eps=(S_\eps(0,x,D_x)g_\eps,g_\eps)+C_\eps\sup_{-T\le t\le T}\Vert f_\eps(t,\cdot)\Vert^2 T.
\]
\subsection{$H^l$-energy estimates for the solution net $(u_\eps)_\eps$}
The $L^2$-positive definite symmetriser $S$ constructed in the previous subsection can be employed to obtain $H^l$-energy estimates for the net $(u_\eps)_\eps$ for any $l\in\R$. Before proceeding with the proof we recall that $H^l$ (and analogously $(H^l)^m$) is a Hilbert space with the scalar product defined by $(u,v)_l=(\lara{D_x}^l I u,\lara{D_x}^l Iv)$, where $\lara{D_x}^l$ is the pseudodifferential operator with symbol $\lara{\xi}^l$. For the sake of simplicity in the sequel we use the short notation $\lara{D_x}^l$ for the matrix $\lara{D_x}^lI$. We now argue on $u_\eps, f_\eps\in \Cinf([-T,T], H^l(\R^n))^m$ and $g_\eps\in {H^l(\R^n)}^m$. The fact that $S_\eps$ is $L^2$-positive definite yields
\[
\Vert u_\eps\Vert_l^2=\Vert \lara{D_x}^{l}u_\eps\Vert^2\le c^{-1}(S_\eps\lara{D_x}^l u_\eps,\lara{D_x}^l u_\eps).
\]
In analogy with \eqref{ee_0} we have
\[
\begin{split}
&\frac{d}{dt}(S_\eps \lara{D_x}^l u_\eps,\lara{D_x}^l u_\eps)=(\partial_t S_\eps \lara{D_x}^l u_\eps,\lara{D_x}^l u_\eps)+(S_\eps\lara{D_x}^l\partial_t u_\eps,\lara{D_x}^l u_\eps)+(S_\eps \lara{D_x}^l u_\eps,\lara{D_x}^l\partial_t u_\eps)\\
&=(\partial_t S_\eps \lara{D_x}^l u_\eps,\lara{D_x}^l u_\eps)+(S_\eps\lara{D_x}^l K_\eps u_\eps +S_\eps\lara{D_x}^l f_\eps,\lara{D_x}^lu_\eps)+(S_\eps\lara{D_x}^l u_\eps,\lara{D_x}^l(K_\eps u_\eps +f_\eps))\\
&\le c_\eps(\partial_t S_\eps)\Vert u_\eps\Vert^2_l+(S_\eps\lara{D_x}^l K_\eps u_\eps,\lara{D_x}^lu_\eps)+ 2\Re\, (S_\eps\lara{D_x}^l f_\eps,\lara{D_x}^lu_\eps)+(S_\eps\lara{D_x}^l u_\eps,\lara{D_x}^l K_\eps u_\eps)\\
&\le c_\eps(\partial_t S_\eps)\Vert u_\eps\Vert^2_l+c_\eps(S_\eps)(\Vert u_\eps\Vert^2_l+\Vert f_\eps\Vert^2_l)+ (S_\eps\lara{D_x}^l K_\eps u_\eps,\lara{D_x}^lu_\eps)+(S_\eps\lara{D_x}^l u_\eps,\lara{D_x}^l K_\eps u_\eps).
\end{split}
\]
Writing $(S_\eps\lara{D_x}^l K_\eps u_\eps,\lara{D_x}^lu_\eps)+(S_\eps\lara{D_x}^l u_\eps,\lara{D_x}^l K_\eps u_\eps)$ as
\begin{multline*}
(S_\eps K_\eps\lara{D_x}^l u_\eps, \lara{D_x}^l u_\eps)+(S_\eps[\lara{D_x}^l,K_\eps]u_\eps,\lara{D_x}^lu_\eps)\\
+(S_\eps\lara{D_x}^lu_\eps,K_\eps\lara{D_x}^l u_\eps)+(S_\eps\lara{D_x}^l u_\eps,[\lara{D_x}^l,K_\eps]u_\eps)
\end{multline*}
we arrive at
\beq
\label{ee_l}
\begin{split}
&\frac{d}{dt}(S_\eps \lara{D_x}^l u_\eps,\lara{D_x}^l u_\eps)\\
&\le c_\eps(\partial_t S_\eps, S_\eps)(\Vert u_\eps\Vert^2_l+\Vert f_\eps\Vert^2_l)+((S_\eps K_\eps+ K_\eps^\ast S_\eps)\lara{D_x}^l u_\eps,\lara{D_x}^lu_\eps)+2\Re\,(S_\eps[\lara{D_x}^l,K_\eps]u_\eps,\lara{D_x}^lu_\eps)\\
&\le C_{l,\eps}(\partial_t S_\eps,S_\eps,S_\eps K_\eps+K_\eps^\ast S_\eps,S_\eps[\lara{D_x}^l,K_\eps])(\Vert u_\eps\Vert^2_l+\Vert f_\eps\Vert^2_l)\\
&\le C_{l,\eps} c^{-1}(S_\eps \lara{D_x}^l u_\eps,\lara{D_x}^lu_\eps)+C_{l,\eps}\Vert f_\eps\Vert^2_l.
\end{split}
\eeq
Returning to the Cauchy problem \eqref{eq_repr}, we obtain by applying the Gronwall lemma to the energy estimates \eqref{ee_l} that
\beq
\label{1_ineq_l}
\Vert u_\eps(t,\cdot)\Vert^2_l\le c^{-1}(S_\eps\lara{D_x}^l u_\eps,\lara{D_x}^l u_\eps)\le c^{-1}d_{l,\eps} \exp({c^{-1}C_{l,\eps} t}),
\eeq
valid for all $\eps\in(0,1]$ and $t\in[-T,T]$ with
\[
d_{l,\eps}=(S_{\eps}(0,x,D_x)\lara{D_x}^l g_\eps,\lara{D_x}^l g_\eps)+C_{l,\eps}\sup_{-T\le t\le T}\Vert f_\eps(t,\cdot)\Vert^2_l T.
\]
and $(C_{l,\eps})_\eps$ depending on the symbol seminorms of $(S_{\eps})_\eps$, $(\partial_t S_{\eps})_\eps$, $(S_{\eps} K_\eps+K_\eps^\ast S_{\eps})_\eps$ and $(S_\eps[\lara{D_x}^l,K_\eps])_\eps$.

\subsection{Choice of scale for the generalised strictly hyperbolic case}
From the estimates \eqref{1_ineq} and \eqref{1_ineq_l} it is clear that the nets $(c^{-1}C_\eps)_\eps$ and $(c^{-1}C_{l,\eps})_\eps$ need to be of logarithmic type in order to generate a Colombeau solution from the net $(u_\eps)_\eps$. This logarithmic behaviour is entailed by the following assumptions on the principal part $K_1(t,x,D_x)$ and the corresponding eigenvalues $\lambda_j$.
\begin{proposition}
\label{prop_log_sc}
Let $K_1(t,x,D_x)$ have symbol entries which are logarithmic slow scale regular. Let the operator $\partial/\partial t-K(t,x,D_x)$ be generalised strictly hyperbolic with eigenvalues $\lambda_j$ such that the bound from below \eqref{bound_fb} is true for some net $(\lambda_\eps)_\eps$, inverse of a logarithmic slow scale net. Let $S$ be the symmetriser constructed in Proposition \ref{prop_symm_S}. Then, the entries of $S$ and $SK+K^\ast S$ are given by $0$-order logarithmic slow scale regular generalised symbols.
\end{proposition}
\begin{proof}
Under the assumptions that $K_{1,\eps}(t,x,\xi)$ and $(\lambda_\eps)_\eps$ are both logarithmic slow scale regular one easily sees that the eigenvalues $(\lambda_{j,\eps})_\eps$, as well as the projectors $P_{j,\eps}(t,x,\xi)$,  $j=1,...,m$, are logarithmic slow scale regular. It follows that the symmetriser $R(t,x,D_x)$, as constructed in the proof of Proposition \ref{prop_symmetriser}, is given by a matrix of generalised symbols which are logarithmic slow scale regular, i.e. symbols in $\G^{\rm{lsc}}_{\Cinf([-T,T], S^0(\R^{2n}))}$. Hence $RK+(RK)^\ast$ has symbols entries in $\G^{\rm{lsc}}_{\Cinf([-T,T], S^0(\R^{2n}))}$ as well. Following the construction of $S$ from $R$ as in the proof of Proposition \ref{prop_symm_S} and by applying G\r{a}rding's inequality of Theorem \ref{theo_Garding_23}$(ii)$ we easily see that $S$ is generated by
\[
S_\eps=R_\eps+c_{1,\eps}\lara{D_x}^{-1}I,
\]
where $(c_{1,\eps})_\eps$ is a logarithmic slow scale net. Hence, it has symbol entries in $\G^{\rm{lsc}}_{\Cinf([-T,T], S^0(\R^{2n}))}$, $S=S^\ast$, $SK+K^\ast S$ is of order $0$ and logarithmic slow scale regular and $(S_\eps u,u)\ge c\Vert u\Vert$ for all $u\in L^2(\R^n)^m$.
\end{proof}
\begin{corollary}
\label{corol_log_sc}
Let $K_1(t,x,D_x)$ have symbol entries which are logarithmic slow scale regular. Let the operator $\partial/\partial t-K(t,x,D_x)$ be generalised strictly hyperbolic with eigenvalues $\lambda_j$ such that the bound from below \eqref{bound_fb} is true for some net $(\lambda_\eps)_\eps$, inverse of a logarithmic slow scale net. Then in the estimate \eqref{1_ineq_l} one has $(C_{l,\eps})_\eps$ of logarithmic slow scale type.
\end{corollary}
\begin{proof}
$(C_{l,\eps})_\eps$ depends on the symbol seminorms of $(S_{\eps})_\eps$, $(\partial_t S_{\eps})_\eps$, $(S_{\eps} K_\eps+K_\eps^\ast S_{\eps})_\eps$ and $(S_\eps[\lara{D_x}^l,K_\eps])_\eps$. From Proposition \ref{prop_log_sc} it follows that $(C_{l,\eps})_\eps$ is a logarithmic slow scale net.
\end{proof}

\section{The Cauchy problem for a first order generalised strictly hyperbolic system}
\label{sec_CP}
We are now able to prove the well-posedness in the Colombeau framework of the generalised strictly hyperbolic Cauchy problem
\beq
\begin{split}
\label{CP}
\frac{\partial}{\partial t}u&=K(t,x,D_x)u+f(t,x),\\
 u|_{t=0}&=g,
\end{split}
\eeq
where
\begin{itemize}
\item[(h1)] $K(t,x,D_x)$ is an $m\times m$-matrix of generalised pseudodifferential operators with classical symbol belonging to $\G^{\rm{lsc}}_{\Cinf([-T,T], S^1(\R^{2n}))}$;
\item[(h2)] $f$ has entries in $\G_{2,2}([-T,T]\times\R^n)$;
%\item[(h3)] the entries of $K(t,x,D_x)$ and $f(t,x)$ are compactly supported in $x$;
\item[(h3)] the operator $\partial/\partial t-K(t,x,D_x)$ is generalised strictly hyperbolic with eigenvalues $\lambda_j$ such that the bound from below \eqref{bound_fb} is true for some net $(\lambda_\eps)_\eps$, inverse of a logarithmic slow scale net;
\item[(h4)] $g\in\G_{2,2}(\R^n)^m$.
\end{itemize}
\begin{theorem}
\label{theo_CP}
Under the hypotheses $(h1)-(h4)$ the Cauchy problem \eqref{CP} has a unique solution $u\in\G_{2,2}((-T,T)\times\R^n)^m$. If $f$ as well as the initial data $g$ are $\Ginf_{2,2}$-regular then $u\in \G^\infty_{2,2}(\R^n)^m$.
\end{theorem}
\begin{proof}
Working at the level of representatives we obtain a net $(u_\eps)_\eps\in H^\infty((-T,T)\times\R^n)^m$ of classical solutions to the Cauchy problem \eqref{CP} (see \cite[Chapter 7, Theorem 3.2]{Kumano-go:81} and \cite[Chapter 5]{RuzhaSug:06}). The Sobolev energy estimates \eqref{1_ineq_l} obtained by means of the symmetriser $S$ in the previous section provide the desired Colombeau moderateness (for the $t$-derivatives one argues by induction directly on $\partial_t u_\eps=K_\eps(t,x,D_x)u_\eps+f_\eps(t,x))$. Finally, the net $(u_\eps)_\eps$ generates a Colombeau solution $u\in\G_{2,2}((-T,T)\times\R^n)^m$. The uniqueness can be easily checked via \eqref{1_ineq_l} in terms of negligibility estimates. From the energy estimates it follows that $u\in \G^\infty_{2,2}(\R^n)^m$ if $f$ and $g$ are $\Ginf_{2,2}$-regular.
\end{proof}
%\begin{remark}
%\label{rem_no_comp}
%When we drop the assumption of compact support with respect to $x$ the Cauchy problem \eqref{CP} becomes well-posed in the Colombeau space $\G_{2,2}((-T,T)\times\R^n)^m:=(\G_{H^\infty((-T,T)\times\R^n)})^m$. This set-up has been already used for the scalar hyperbolic pseudodifferential equations in \cite[Section 3]{GO:10a}. In detail, if
%\begin{itemize}
%\item[-] $K(t,x,D_x)$ is a $m\times m$-matrix of generalised pseudodifferential operators with classical symbol belonging to $\G^{\rm{lsc}}_{\Cinf([-T,T], S^1(\R^{2n}))}$;
%\item[-] $f$ has entries in $\G_{2,2}([-T,T]\times\R^n)$;
%\item[-] the operator $\partial/\partial t-K(t,x,D_x)$ is generalised strictly hyperbolic with eigenvalues $\lambda_j$ such that the bound from below \eqref{bound_fb} is true for some net $(\lambda_\eps)_\eps$, inverse of a logarithmic slow scale net;
%\item[-] $g\in\G_{2,2}(\R^n)^m$.
%\end{itemize}
%then there exists a unique solution $u\in\G_{2,2}((-T,T)\times\R^n)^m$ to the Cauchy problem \eqref{CP}. One easily sees that $u\in\G^\infty_{2,2}((-T,T)\times\R^n)^m$ when $f$ has entries in $\G^\infty_{2,2}((-T,T)\times\R^n)$ and $g$ belongs to $\G^\infty_{2,2}(\R^n)^m$.
%\end{remark}

\section{Higher order hyperbolic equations with generalised coefficients: reduction to a generalised strictly hyperbolic system}
\label{sec_system}
The aim of this section is to study the Cauchy problem for hyperbolic equations of the following type:
\beq
\label{high_hyp_CP}
Lu=f,\qquad \frac{\partial^j}{\partial t^j}u|_{t=0}=g_{j+1},\quad j=0,...,m-1,
\eeq
where
\beq
\label{type_L_m}
L=\frac{\partial^m}{\partial t^m}-\sum_{j=0}^{m-1}A_{m-j}(t,x,D_x)\frac{\partial^j}{\partial t^j}
\eeq
and each $A_{m-j}$ is a differential operator of order $m-j$ with Colombeau coefficients (the Colombeau space type will be specified in the sequel).

We begin by performing a reduction to a first order system as in \cite{Taylor:81}. Let  $\lara{D_x}$ be the pseudodifferential operator with symbol $\lara{\xi}$. The transformation
\[
u_j=\frac{\partial^{j-1}}{\partial t^{j-1}}\lara{D_x}^{m-j}u,
\]
with $j=1,...,m$, makes the Cauchy problem \eqref{high_hyp_CP} equivalent to the following system
\beq
\label{syst_Taylor}
\frac{\partial}{\partial t}\left(
                             \begin{array}{c}
                               u_1 \\
                               \cdot \\
                               \cdot\\
                               u_m \\
                             \end{array}
                           \right)
= \left(
    \begin{array}{ccccc}
      0 & \lara{D_x} & 0 & \dots & 0\\
      0 & 0 & \lara{D_x} & \dots & 0 \\
      \dots & \dots & \dots & \dots & \lara{D_x} \\
      b_1 & b_2 & \dots & \dots & b_m \\
    \end{array}
  \right)
\left(\begin{array}{c}
                               u_1 \\
                               \cdot \\
                               \cdot\\
                               u_m \\
                             \end{array}
                           \right)
                           +\left(
                              \begin{array}{c}
                                0 \\
                                \cdot \\
                                \cdot \\
                                f \\
                              \end{array}
                            \right),
\eeq
where
\[
b_j=A_{m-j+1}(t,x,D_x)\lara{D_x}^{j-m},
\]
with initial condition
\beq
\label{ic_Taylor}
u_j|_{t=0}=\lara{D_x}^{m-j}g_j,\qquad j=1,...,m.
\eeq
The generalised strict hyperbolicity of the operator $L$ in \eqref{type_L_m} can be equivalently stated in terms of the roots $\tau(t,x,\xi)$ of
\[
P(\tau,t,x,\xi):=(i\tau)^m-\sum_{j=0}^{m-1}\wt{A}_{m-j}(t,x,\xi)(i\tau)^j,
\]
where $\wt{A}_{m-j}(t,x,\xi)$ is the principal part of ${A}_{m-j}(t,x,\xi)$, or in terms of the eigenvalues of the matrix
\[
B(t,x,\xi):=\left(
    \begin{array}{ccccc}
      0 & \lara{\xi} & 0 & \dots & 0\\
      0 & 0 & \lara{\xi} & \dots & 0 \\
      \dots & \dots & \dots & \dots & \lara{\xi} \\
      \wt{b_1} & \wt{b_2} & \dots & \dots & \wt{b_m} \\
    \end{array}
  \right),
\]
where $\wt{b_j}(t,x,\xi)=\wt{A}_{m-j+1}(t,x,\xi)\lara{\xi}^{j-m}$.
\begin{definition}
\label{def_sh_m}
The operator $L$ in \eqref{type_L_m} is called \emph{generalised strictly hyperbolic} if there exists a choice of representatives of the coefficients in $\G([-T,T]\times\R^n)$ such that the corresponding polynomial $P_\eps(\tau,t,x,\xi)$ has $m$ distinct real valued roots $(\tau_{j,\eps})_\eps$ such that
\beq
\label{bound_fb_m}
\tau_{j+1,\eps}(t,x,\xi)-\tau_{j,\eps}(t,x,\xi)\ge \tau_\eps\lara{\xi}
\eeq
holds for some strictly nonzero net $(\tau_\eps)_\eps$, for all $j=1,...,m-1$, for all $\eps\in(0,1]$, $t\in[-T,T]$, $x\in\R^n$ and for $\xi$ away from zero.
\end{definition}
For the same choice of the coefficient representatives we have that the purely imaginary nets $(i\tau_{j,\eps})_\eps$, $j=1,...,m$ are eigenvalues of the matrix $B_\eps$ defined above and that Definition \ref{def_sh_m} coincides with Definition \ref{def_sh}.

The well-posedness result of Section \ref{sec_CP} for the Cauchy problem \eqref{syst_Taylor}-\eqref{ic_Taylor} yields the following statement. Note that for the Sylvester matrix in \eqref{syst_Taylor} we use the set-up of Theorem \ref{theo_CP}.
\begin{theorem}
\label{theo_CP_m}
Let $L$ be a differential operator of order $m$ as in \eqref{type_L_m} with coefficients in $\G_{\rm{b}}([-T,T]\times\R^n)$ such that its symbol belongs to $\G^{\rm{lsc}}_{\Cinf([-T,T], S^m(\R^{2n}))}$. Let $f\in\G_{2,2}((-T,T)\times\R^n)$ and $g_{j+1}\in\G_{2,2}(\R^n)$ for $j=0,...,m-1$. Let $L$ be generalised strictly hyperbolic with roots $\tau_j$ satisfying the bound from below \eqref{bound_fb_m} for some net $(\tau_\eps)_\eps$, inverse of a logarithmic slow scale net. Then there exists a unique solution $u\in\G_{2,2}((-T,T)\times\R^n)$ to the Cauchy problem \eqref{high_hyp_CP}. Furthermore, if $f\in\G^{\infty}_{2,2}((-T,T)\times\R^n)$ and the initial data $g_{j+1}$ belong to $\Ginf_{2,2}(\R^n)$ then $u\in\G^{\infty}_{2,2}((-T,T)\times\R^n)$.
\end{theorem}
\begin{remark}
\label{rem_CP_m}
The theorem holds, in particular, for logarithmic slow scale regular coefficients that are constant for large $x$ as well as for logarithmic slow scale regular coefficients depending on $t$ only, or belonging to $\G_{2,2}((-T,T)\times\R^n)$.
\end{remark}

\renewcommand{\supp}{\mathop{\mathrm{supp}}}
\renewcommand{\WF}{\mathop{\mathrm{WF}}}
\renewcommand{\WFg}{\mathrm{WF}_{\mathrm{g}}\hspace{1pt}}
\renewcommand{\grad}{\operatorname{grad}}
\renewcommand{\div}{\operatorname{{div}}}
\renewcommand{\singsupp}{\mathop{\mathrm{sing\,supp}}}

\newcommand{\p}{\partial}
\newcommand{\Hdiv}{H_{\rm div}}

\section{Applications and limiting solutions}

In the previous sections we have obtained generalised solutions to the Cauchy problem for hyperbolic equations and systems with generalised coefficients. As mentioned in the Introduction, classical solutions to the Cauchy problem may exist in certain sufficiently regular cases. It is the purpose of this section to show how the generalised solution in the Colombeau algebra relates to such a classical or piecewise classical solution, if it exists. We will show exemplarily (for the multidimensional system of linear acoustics and the wave equation) that the generalised solution is associated with the classical solution or the piecewise classical solution, if the appropriate transmission conditions are imposed. I fact, in physical systems (e.g., in acoustics), the Colombeau solution encodes already the physically meaningful transmission condition and selects the corresponding piecewise classical solution in the limit. A similar result for a boundary value problem in one-dimensional acoustics has been obtained in \cite{O:89}.

As one of the central examples of interest in applications we consider the acoustics system. It describes small perturbations from a
state of rest of the velocity $v(x,t)\in\R^n$, the pressure $p(x,t)\in\R$ and the density $\rho(x,t)\in\R$ of a substance in space ($x\in \R^n$) that moves in a medium of given density $\rho_0(x)$ and sound speed $c_0(x)$. It can be derived from the Euler equations of isentropic gas dynamics, written with material time derivative \cite{Whitham:74}. We will be particularly interested in the case where $\rho_0$ and $c_0$ have a jump across some hypersurface, signifying a change of material
properties at a separating interface.

The acoustics system is
\begin{eqnarray}\label{eq:acousticsgen}
\frac{\p}{\p t} \rho + \rho_0(x)\div v & = & 0 \nonumber\\
\rho_0(x)\frac{\p}{\p t} v + \grad p & = & 0\\
p & = & c_0^2(x)\,\rho \nonumber
\end{eqnarray}
The divergence and gradient operators are understood with respect to the variables $x\in\R^n$.
We are assuming that $\rho_0$ and $c_0$ do not depend on $t$ and are bounded away from zero. Thus we can cast the system in the form
\begin{equation}\label{eq:acoustics}
\begin{array}{rcl}
\displaystyle\frac{\p}{\p t} p + c_0^2(x)\rho_0(x)\div v & = & 0\\[8pt]
\displaystyle\frac{\p}{\p t} v + \frac{1}{\rho_0(x)}\grad p & = & 0
\end{array}
\end{equation}
We adjoin initial conditions
\begin{equation}\label{eq:ICacoustics}
p(x,0) = p_0(x),\qquad v(x,0) = v_0(x).
\end{equation}
For sufficiently smooth solutions, the system (\ref{eq:acoustics}) with initial conditions (\ref{eq:ICacoustics}) is equivalent to the
wave equation
\begin{equation}\label{eq:wave}
   \frac{\p^2}{\p t^2} p - c_0^2(x)\rho_0(x)\div\Big(\frac1{\rho_0(x)}\grad p\Big) = 0
\end{equation}
with initial conditions
\begin{equation}\label{eq:ICwave}
p(x,0) = p_0(x),\qquad \frac{\p}{\p t} p(x,0) = - c_0^2(x)\rho_0(x)\div v_0(x).
\end{equation}
Indeed, differentiating (\ref{eq:acoustics}) and subtracting the second from the first line immediately gives (\ref{eq:wave}). For the other direction,
one sets
\[
  \frac{\p}{\p t} v = - \frac{1}{\rho_0(x)}\grad p, \qquad v(x,t) = v_0(x) + \int_0^t\frac{1}{\rho_0(x)}\grad p(x,s)\,ds
\]
and uses the second initial condition (\ref{eq:ICwave}) to deduce the first line in (\ref{eq:acoustics}) from (\ref{eq:wave}).

We will see that the equivalence of the acoustics system and the wave equation pertains for solutions in the Colombeau algebra, hence we may freely switch from one formulation to the other. We will first consider the acoustics system (\ref{eq:acoustics}) in any space dimension and then derive some more specific results for the wave equation
(\ref{eq:wave}) in one  space dimension.

\subsection{The acoustics system in any space dimension}

Writing the system (\ref{eq:acoustics}) in the form
\[
   \frac{\p}{\p t} u = K(t,x,D_x)u,\qquad u|_{t=0} = u_0
\]
for the lumped vector $u = (p,v_1,\ldots,v_n)$, we find that the symbol matrix is
\[
  K(t,x,\xi) = \begin{bmatrix}
               0 & -i\rho_0(x)c_0^2(x)\xi_1 & \cdots & -i\rho_0(x)c_0^2(x)\xi_n\\[4pt]
               \frac{-i}{\rho_0(x)}\xi_1 & 0                        & \cdots & 0\\[4pt]
               \vdots & \vdots                   &        & \vdots                   \\[4pt]
               \frac{-i}{\rho_0(x)}\xi_n & 0                        & \cdots & 0
                \end{bmatrix}
\]
Its eigenvalues are $\lambda=\pm ic_0(x)\sqrt{\xi_1^2 + \cdots + \xi_n^2}$ and $\lambda = 0$, the latter with
multiplicity $n-1$. Thus the system is not strictly hyperbolic for $n > 2$. For $n\leq 2$, Proposition~\ref{prop_symmetriser}
yields symmetrisability in the Colombeau setting. For arbitrary $n \geq 1$, it has been shown in \cite{LO:91} that
the system can be symmetrized by a change of dependent variables. Nevertheless, in order to remain in the framework of the previous sections, we will turn to the wave equation formulation (\ref{eq:wave}) for further analysis. Writing (\ref{eq:wave})
as
\[
   \frac{\p^2}{\p t^2} p - c_0^2(x)\Delta p - c_0^2(x)\rho_0(x)\grad\frac{1}{\rho_0(x)}\cdot\grad p = 0
\]
we see that the principal symbol $P(\tau,t,x,\xi)$ has the roots $\tau(t,x,\xi) = \pm c_0(x)|\xi|$, which are real and distinct since $c_0$ was assumed to be bounded away from zero.

The results of Section~\ref{sec_system} can now be applied to the wave equation~(\ref{eq:wave}) and the acoustics system~(\ref{eq:acoustics}).

\begin{theorem}
\label{thm:acousticwave}
Let $c_0, \rho_0\in\G_{\rm{b}}(\R^n)$ be logarithmic slow scale regular generalised functions, bounded from below by the inverse of a logarithmic slow scale net. Let $p_0, v_0\in\G_{2,2}(\R^n)$ and $T>0$. Then the initial value problem (\ref{eq:wave}), (\ref{eq:ICwave}) has a unique solution $p\in\G_{2,2}((-T,T)\times\R^n)$, and the initial value problem (\ref{eq:acoustics}), (\ref{eq:ICacoustics}) has a unique solution $(p,v) \in\G_{2,2}((-T,T)\times\R^n)^{n+1}$.
\end{theorem}
\begin{proof}
The existence of the solution $p$ to (\ref{eq:wave}), (\ref{eq:ICwave}) follows immediately from Theorem~\ref{theo_CP_m}. Under the assumptions on $c_0,\rho_0$, it is clear that the acoustics system (\ref{eq:acoustics}), (\ref{eq:ICacoustics}) is equivalent to the wave equation (\ref{eq:wave}), (\ref{eq:ICwave}) in $\G_{2,2}((-T,T)\times\R^n)$. Thus the second assertion holds as well.
\end{proof}
\begin{remark}\label{rem:gen-class}
To make the connection with possible classical solutions, consider coefficients $\overline{c}_0, \overline{\rho}_0\in L^\infty(\R^n)$ which are essentially bounded away from zero by some positive real constant. Take a compactly supported net $(\varphi_\eps)_\eps$ of nonnegative functions which is logarithmic slow scale regular and converges to the Dirac measure. Then $c_{0\eps} = \overline{c}_0\ast \varphi_\eps$ and $\rho_{0,\eps} = \overline{\rho}_0\ast\varphi_\eps$ are representatives of generalised functions $c_0, \rho_0\in\G(\R^n)$ that satisfy the hypotheses of Theorem~\ref{thm:acousticwave}.
\end{remark}

In order to prove the convergence of the representatives of the generalised solution to a classical solution (when it exists), we need to make a digression into the classical theory. In order not to overload the notation, we will use the letters $v$, $p$ etc. in system (\ref{eq:acoustics}) temporarily for classical solutions. We assume that $c_0$ and $\rho_0$ are measurable functions satisfying bounds from below and above:
\begin{equation}\label{eq:bounds}
   0 < \rho_1 \leq \rho_0(x) \leq \rho_2,\qquad 0 < c_1 \leq c_0(x) \leq c_2.
\end{equation}
We are going to show that the operator defined by (\ref{eq:acoustics}) generates a semigroup of type $(M,0)$ where $M$ depends only on the four bounding constants in (\ref{eq:bounds}). As usual, we let
\[
   \Hdiv(\R^n) = \{v\in\big(L^2(\R^n)\big)^n:\div v \in L^2(\R^n)\}.
\]
It is the domain of the divergence operator, which is a closed operator on $\big(L^2(\R^n)\big)^n$.
We equip the Hilbert space
\[
   X = L^2(\R^n) \times \big(L^2(\R^n)\big)^n
\]
with the equivalent inner product $(\cdot|\cdot)$ arising from the weighted norm
\[
   \|(p,v)\|_{c_0\rho_0}^2 = \int_{\R^n} |p(x)|^2\frac{1}{c_0^2(x)\rho_0(x)}\,dx + \int_{\R^n} |v(x)|^2 \rho_0(x)\,dx
\]
The unbounded operator $A$ on $X$ is defined by
\[
   D(A) = H^1(\R^n) \times \Hdiv(\R^n),\qquad A(p,v) = (-c_0^2\rho_0\div v, -\rho_0^{-1}\grad p).
\]
In this setting, problem (\ref{eq:acoustics}) reads $\frac{d}{dt}(p,v) = A(p,v)$. The following assertion is the heart of the matter.
\begin{proposition}\label{prop:semigroup}
The operator $A$ generates a contraction semigroup $\{S(t):t\geq 0\}$ on $X$.
\end{proposition}
\begin{proof}
The proof is standard, so we just sketch the ingredients. First one shows by regularisation and cut-off that $\big({\mathcal C}^\infty_{\rm c}(\R^n)\big)^n$ is dense in $\Hdiv(\R^n)$, equipped with the graph norm of the divergence operator. In particular, $A$ is densely defined. It is also quite obvious that $A$ is closed. Next,
\begin{equation}\label{eq:conservativestep}
   \int_{\R^n}\big( p(x)\div v(x) + \grad p(x) \cdot v(x)\big)dx = 0
\end{equation}
for $(p,v) \in D(A)$. This follows from the fact that the left-hand side in (\ref{eq:conservativestep}) equals
$\int_{\R^n} \div(p(x)v(x))dx$, which vanishes for $(p,v) \in \big({\mathcal C}^\infty_{\rm c}(\R^n)\big)^{n+1}$.

The goal of the proof is to show that the operator $A$ is maximal dissipative on $X$. As a first step, we observe that $A$ is conservative, that is, $(Au|u) = 0$ for all $u = (p,v)\in D(A)$. Indeed,
\[
   (Au|u) = - \int_{\R^n} c_0^2(x)\rho_0(x)\div v(x)\frac{1}{c_0^2(x)\rho_0(x)}\,dx
            - \int_{\R^n} \frac1{\rho_0(x)}\grad p(x)\cdot v(x)\,\rho_0(x)\,dx = 0
\]
by \eqref{eq:conservativestep}. In particular, $A$ is dissipative, i.e. $(Au|u) \leq 0$ for all $u = (p,v)\in D(A)$. The fact that $A$ is skew-adjoint, that is, $D(A^\ast) = D(A)$ and $A^\ast = -A$ follows from standard arguments and will be omitted. Thus $A^\ast$ is  conservative as well, hence dissipative. Recall that an operator $B$ is dissipative if and only if
\[
    \|(\lambda - B)u\| \geq \lambda\|u\|\quad\mbox{for\ all\ } u\in D(B) \mbox{\ and\ }\lambda > 0.
\]
This implies that $(\lambda - B)$ is injective and has a closed range for all $\lambda > 0$. We apply this to $A$ and $A^\ast$ and infer that $(\lambda - A)$ has a closed range and $(\lambda - A^\ast)$ is injective. The latter implies that the range of $(\lambda - A)$ is dense. Combining these observations, we conclude that the range of $(\lambda - A)$ equals $X$ for all $\lambda > 0$. This is well-known to be equivalent to the property that $A$ is maximal dissipative.

Further, $(\lambda - A): D(A)\to X$ is bijective and we have the resolvent estimate
\[
   \|(\lambda - A)^{-1}\| \leq \frac1\lambda\quad\mbox{for\ all\ } \lambda > 0.
\]
The Hille-Yosida-Lumer-Phillips theorem (cf. e.g. \cite{Pazy:83}) implies that $A$ generates a semigroup of contractions on $X$.
\end{proof}

In particular, this means that the operator norm $\|S(t)\|_{c_0\rho_0} \leq 1$ for all $t\geq 0$ with respect to the weighted norm on $X$. We wish to compute the operator norm $\|S(t)\|$ with respect to the usual norm on $L^2(\R^n) \times \big(L^2(\R^n)\big)^n$. It is clear that
\[
   C_1\|(p,v)\|_{c_0\rho_0} \leq \|(p,v)\| \leq C_2\|(p,v)\|_{c_0\rho_0}
\]
for some constants $C_1, C_2$ depending only on $c_1, c_2, \rho_1, \rho_2$. From there it follows that
\begin{equation}\label{eq:typeM}
\|S(t)\| \leq \frac{C_2}{C_1}\|S(t)\|_{c_0\rho_0} \leq \frac{C_2}{C_1} = M
\end{equation}
for all $t\geq 0$. In particular, $\{S(t):t\geq 0\}$ is a semigroup of type $(M,0)$ on $L^2(\R^n) \times \big(L^2(\R^n)\big)^n$.

\begin{corollary}\label{cor:class-solu}
Let $\overline{p}_0\in H^1(\R^n)$, $\overline{v}_0\in \Hdiv(\R^n)$ and $c_0, \rho_0$ as in \eqref{eq:bounds}. Then problem \eqref{eq:acoustics} has a unique solution
$(\overline{p},\overline{v})$ in ${\mathcal C}^1\big([0,\infty):L^2(\R^n) \times (L^2(\R^n)^n\big)$ such that $(\overline{p}(t),\overline{v}(t))$ belongs to
$H^1(\R^n) \times \Hdiv(\R^n)$ for all $t\geq 0$, namely $(\overline{p}(t),\overline{v}(t)) = S(t)(\overline{p}_0,\overline{v}_0)$.
\end{corollary}
\begin{proof}
This is standard semigroup theory \cite{Pazy:83}.
\end{proof}
We are now in the position to relate the generalised solution with the classical solution, when it exists.
\begin{proposition}
Assume that $\overline{c}_0,\overline{\rho}_0\in L^\infty(\R^n)$ satisfy the bounds \eqref{eq:bounds}. Let $c_0,\rho_0\in {\mathcal G}_{\rm b}(\R^n)$ be corresponding generalised functions as defined in Remark~\ref{rem:gen-class}. Let $\overline{p}_0\in H^1(\R^n)$, $\overline{v}_0\in \Hdiv(\R^n)$ and $p_0, v_0\in\G_{2,2}(\R^n)$ be corresponding generalised functions with representatives $(p_{0\eps}, v_{0\eps})$ obtained by convolution with a delta net. Let $T>0$. Then the unique solution $(p,v) \in\G_{2,2}((-T,T)\times\R^n)^{n+1}$ to \eqref{eq:acoustics}, \eqref{eq:ICacoustics} given in Theorem~\ref{thm:acousticwave} is associated with the classical solution
$(\overline{p}(t),\overline{v}(t))$ given in Corollary~\ref{cor:class-solu}, that is, its representatives converge to the classical solution as $\eps\to 0$.
\end{proposition}
\begin{proof}
As can be seen from Remark~\ref{rem:gen-class}, the representatives $c_{0\eps}$, $\rho_{0\eps}$ of the generalised functions $c_0,\rho_0$ satisfy bounds of the type \eqref{eq:bounds} with $c_1, c_2, \rho_1,\rho_2$ independent of $\eps$, due to the nonnegativity of the mollifier. Taking $c_{0\eps}$, $\rho_{0\eps}$ as coefficients in \eqref{eq:acoustics}, Proposition~\ref{prop:semigroup} guarantees that the corresponding operator $A_\eps$ generates of a semigroup $\{S_\eps(t):t\geq 0\}$. Due to the classical and the generalised uniqueness result, the solution $(p_\eps(\cdot, t),v_\eps(\cdot, t)) = S_\eps(t)(p_{0\eps}, v_{0\eps})$ is a representative of the generalised solution $(p,v)$. The estimate \eqref{eq:typeM} shows that the semigroups $\{S_\eps(t):t\geq 0\}$ are of type $(M,0)$ on the common Hilbert space $L^2(\R^n) \times \big(L^2(\R^n)\big)^n$, independently of $\eps$.

We now invoke Kato's perturbation result \cite{Kato:66} which says that under these conditions, $S_\eps(t)(p_0,v_0)$ converges to $S(t)(p_0,v_0)$ uniformly for $t$ in compact time intervals, for every $(p_0,v_0)\in L^2(\R^n) \times \big(L^2(\R^n)\big)^n$, provided the resolvents $(\lambda - A_\eps)^{-1}$ converge strongly to $(\lambda - A)^{-1}$ for some $\lambda > 0$.

The resolvent convergence can be seen as follows. First take $u = (p,v) \in D(A) = D(A_\eps) \in H^1(\R^n)\times \Hdiv(\R^n)$. Then we have for the $L^2$-norm
\[
   \|A_\eps u - A u\|^2 = \int_{\R^n} \big|(c_{0\eps}^2\rho_{0\eps} - \c_0^2\rho_0)\div v\big|^2dx + \int_{\R^n} \big|(\rho_{0\eps}^{-1} - \rho_0^{-1})\grad p\big|^2dx.
\]
Due to the bounds \eqref{eq:bounds} and a wellknown property of Friedrichs mollifiers, the difference of the coefficients in the integrals converges to zero at every Lebesgue point, hence almost everywhere. By Lebesgue's theorem, $\|A_\eps u - A u\|^2 \to 0$ as $\eps\to 0$. Now let $u = (p,v) \in L^2(\R^n) \times \big(L^2(\R^n)\big)^n$ be arbitrary. Then
\[
   \|(\lambda - A_\eps)^{-1}u - (\lambda - A)^{-1}u\| = \|(\lambda - A_\eps)^{-1}\big(u - (\lambda - A_\eps)(\lambda - A)^{-1}u\big)\|.
\]
Since $(\lambda - A)^{-1}u \in D(A)$, the second factor converges to zero by the first step. Since $\{S_\eps(t):t\geq 0\}$ is a semigroup of type $(M,0)$, the resolvent norm $\|(\lambda - A_\eps)^{-1}\|$ is bounded by $M/\lambda$, uniformly in $\eps$. This entails that $\|(\lambda - A_\eps)^{-1}u - (\lambda - A)^{-1}u\|$ converges to zero, as desired.

Finally, the fact that $S_\eps(t)(p_{0\eps},v_{0\eps})$ converges to $S(t)(p_0,v_0)$ as well follows readily from Kato's perturbation result for fixed arguments and the fact that $\|S_\eps(t)\| \leq M$ for all $t\geq 0$.
\end{proof}

\begin{remark}
The classical solution $(\overline{p},\overline{v})$ obtained in Corollary~\ref{cor:class-solu} belongs to $H^1(\R^n) \times \Hdiv(\R^n)$ for all times $t\geq 0$. If in addition it is piecewise continuously differentiable, then one can show that it satisfies the classical transmission conditions across the hypersurface of discontinuity. More precisely, let $\Gamma$ be a smooth hypersurface that splits $\R^n$ into two parts $\Omega_-$, $\Omega_+$. Let $\nu$ be the unit normal to $\Gamma$. One can show the following: If $(\overline{p},\overline{v})$ belongs to ${\mathcal C}^1(\overline{\Omega}_-) \oplus {\mathcal C}^1(\overline{\Omega}_+)$, then $\overline{p}$ and the normal component
$\nu\cdot \overline{v}$ are continuous across $\Gamma$. The second equation in \eqref{eq:acoustics} then implies that also $\frac{1}{\rho_0}\grad \overline{p}$ is continuous, which is the relevant transmission condition in the  wave equation formulation \eqref{eq:wave}.

\end{remark}

\subsection{The one-dimensional wave equation with discontinuous coefficients}

We turn to the wave equation in one space dimension, which -- slightly more general as in the previous subsection -- we take of the form
\begin{equation}\label{eq:wave1D}
   a(x,t)\frac{\p^2}{\p t^2} w - \frac{\p}{\p x}\left(b(x,t)\frac{\p}{\p x} w \right) = 0
\end{equation}
with initial conditions
\begin{equation}\label{eq:ICwave1D}
   w(x,0) = w_0(x),\qquad \frac{\p}{\p t} w(x,0) = w_1(x).
\end{equation}
The functions $a$ and $b$ are assumed to be strictly positive. We will treat the two cases where $a$ and $b$ depend on $x$ or on $t$ only and are piecewise constant. We begin with the case
\begin{equation}\label{eq:constants}
   a(x,t) \equiv \overline{a}(x) = \left\{\begin{array}{ll}
                                                a_-, & x < 0,\\
                                                a_+, & x > 0,
                                          \end{array}\right.
   \qquad
   b(x,t) \equiv \overline{b}(x) = \left\{\begin{array}{ll}
                                                b_-, & x < 0,\\
                                                b_+, & x > 0,
                                          \end{array}\right.
\end{equation}
where the constants $a_-, a_+, b_-, b_+$ are strictly positive.

\begin{remark}\label{rem:class-conn}
If the initial function $\overline{w}_0$ is continuously differentiable and $\overline{w}_1$ is continuous, and $a, b$ are piecewise constant as in \eqref{eq:constants}, then there is a unique distribution $\overline{w} \in {\mathcal D}'(\R^2)$ with the following properties:
\begin{itemize}
\item[-] $\overline{w}$ is a distributional solution to \eqref{eq:wave1D} on the open halfplanes $\{x < 0\}$ and $\{x > 0\}$;
\item[-] $\overline{w} \in {\mathcal C}\big(\R\times [0,\infty))\big)$ and $\overline{w}(\cdot, 0) = \overline{w}_0$;
\item[-] $\frac{\p}{\p t}\overline{w} \in {\mathcal C}\big(\R\times [0,\infty))\big)$ and $\frac{\p}{\p t}\overline{w}(\cdot, 0) = \overline{w}_1$;
\item[-] the function $x\to\overline{b}(x)\frac{\p}{\p x} \overline{w}(x,t)$ is continuous for almost all $t\geq 0$.
\end{itemize}
This is easily seen by transforming \eqref{eq:wave1D} to a first order system on either side of $x = 0$ and joining the two parts of the solution by the transmission conditions that
$\overline{w}$ (and hence $\frac{\p}{\p t}\overline{w}$) as well as $\overline{b}\frac{\p}{\p x}\overline{w}$ should be continuous across $x = 0$.
We call $\overline{w}$ the \emph{classical connected solution} of the transmission problem for the wave equation~\eqref{eq:wave1D}.
\end{remark}
As in the previous subsection, we produce logarithmic slow scale regular generalised functions $a$ and $b$ by means of representatives $a_\eps = \overline{a}\ast\varphi_\eps$ with a mollifier as in Remark~\ref{rem:gen-class}, and similarly for $b_\eps$. By Theorem~\ref{theo_CP_m}, problem \eqref{eq:wave1D} with initial data $w_0, w_1 \in \G_{2,2}(\R)$ has a unique solution $w$ in $\G_{2,2}((-T,T)\times\R)$.

Specifically, we take initial data $\overline{w}_0$ and $\overline{w}_1$ as in Remark~\ref{rem:class-conn}, which in addition -- for the sake of the argument -- have compact support. With any moderate mollifier $\psi_\eps$ converging to the Dirac measure, $w_{0,\eps} = \overline{w}_0\ast\psi_\eps$ defines an element of $\G_{2,2}(\R)$, and similarly with $w_{1\eps}$. Let $w\in\G_{2,2}((-T,T)\times\R)$ be the corresponding generalised solution. It vanishes if $x$ is outside some compact set $K$.
\begin{proposition}
In the situation described above, the generalised solution $w$ to \eqref{eq:wave1D} is associated with the classical connected solution $\overline{w}$.
\end{proposition}
\begin{proof}
Let $(w_\eps)_\eps$ be a representative of the generalised solution that vanishes for $x$ outside $K$. This and the smoothness implies that $w_\eps$ belongs to
${\mathcal C}^\infty\big([0,T):H^\infty(\R)\big)$. Thus we can use energy estimates. Multiplying the wave equation $aw_{\eps tt} - (bw_{\eps x})_x = 0$ by $w_{\eps t}$ and integrating we get the conservation law
\[
   \frac12\frac{d}{dt}\int_\R\left(a|w_{\eps t}|^2 + b |w_{\eps x}|^2\right)dx = 0\quad\mbox{for\ } t\in[0,T).
\]
Differentiating the wave equations once with respect to time, multiplying by $w_{\eps tt}$ and integrating we get
\[
   \frac12\frac{d}{dt}\int_\R\left(a|w_{\eps tt}|^2 + b |w_{\eps xt}|^2\right)dx = 0\quad\mbox{for\ } t\in[0,T).
\]
The first conservation law implies that the family $(w_\eps)_\eps$ is bounded in ${\mathcal C}\big([0,T):H^1(\R)\big)$. This implies that $(w_\eps(\cdot,t))_\eps$ is relatively compact in ${\mathcal C}(\R)$ for every $t$. Using the estimate
\[
    |w_\eps(y,t) - w_\eps(x,t)| \leq \sqrt{|y-x|}\left(\int_\R |\frac{\p}{\p x}w_\eps(y,t)|\,dx\right)^{1/2}
\]
we get the equicontinuity of $(w_\eps(\cdot,t))_\eps$. By Ascoli's theorem, $(w_\eps)_\eps$ is relatively compact in
${\mathcal C}\big([0,T):{\mathcal C}(\R)\big) = {\mathcal C}\big(\R\times[0,T)\big)$. In the same way, the second conservation law shows that $(w_{\eps t})_\eps$ is relatively compact in
${\mathcal C}\big(\R\times[0,T)\big)$. Further, $(w_{\eps tt})_\eps$ and $(w_{\eps x})_\eps$ are bounded in $L^2(\R\times(0,T))$. Thus for a suitable subsequence, we have that there are
$W, W_1 \in {\mathcal C}\big(\R\times[0,T)\big)$ and $W_2, W_3 \in L^2(\R\times(0,T))$ such that
\begin{itemize}
\item[] $w_\eps \to W$, $w_{\eps t} \to W_1$ strongly in ${\mathcal C}\big(\R\times[0,T)\big)$ and
\item[] $w_{\eps tt} \to W_2$, $w_{\eps x} \to W_3$ weakly in $L^2(\R\times(0,T))$.
\end{itemize}
It follows that $W_1 = W_t$, $W_2 = W_{tt}$ and $W_3 = W_x$. The wave equation now yields that $(b_\eps w_{\eps x})_x$ converges weakly to some $W_4\in L^2(\R\times(0,T))$. Due to the support properties, the convolution $V(\cdot,t) = W_4(\cdot,t)\ast H$ exists, where $H$ denotes the Heaviside function. Then $V(\cdot,t)$ is continuous for almost all $t$, and
$b_\eps w_{\eps x} \to V$ weakly in $L^2(\R\times(0,T))$. It follows that
\[
   \overline{a} W_{tt} - V_x = 0.
\]
On the other hand, $b_\eps(x) \equiv b_+$ at every given $x>0$ and $b_\eps(x) \equiv b_-$ at every given $x<0$ for sufficiently small $\eps$. Thus
\begin{equation*}
   V(\cdot,t) = \left\{\begin{array}{ll}
                   b_-W_x, & x < 0,\\
                   b_+W_x, & x > 0,
                   \end{array}\right.
   \qquad
   V_x(\cdot,t) = \left\{\begin{array}{ll}
                   b_-W_{xx}, & x < 0,\\
                   b_+W_{xx}, & x > 0,
                   \end{array}\right.
\end{equation*}
for almost all $t$. It follows that $W$ solves the wave equation on either side of $x=0$, $W$ and $W_t$ are continuous, take the given initial values $\overline{w}_0$ and $\overline{w}_1$, and $\overline{b} W_x(\cdot,t)$ is continuous for almost all $t$. By uniqueness, $W$ is equal to the classical connected solution $\overline{w}$, and the whole net $w_\eps$ converges to it.
\end{proof}
As a final case, we consider the wave equation \eqref{eq:wave1D} with coefficients that have a jump in time, say at $t = 1$. Thus we let
\begin{equation}\label{eq:timeconstants}
   a(x,t) \equiv \overline{a}(t) = \left\{\begin{array}{ll}
                                                a_-, & t < 1,\\
                                                a_+, & t > 1,
                                          \end{array}\right.
   \qquad
   b(x,t) \equiv \overline{b}(t) = \left\{\begin{array}{ll}
                                                b_-, & t < 1,\\
                                                b_+, & t > 1.
                                          \end{array}\right.
\end{equation}
It is obvious that, given initial data ${\mathcal D}'(\R)$, there is a unique distribution
$w \in {\mathcal C}^1\big([0,1]:{\mathcal D}'(\R)\big) \oplus {\mathcal C}^1\big([1,\infty):{\mathcal D}'(\R)\big)$ that satisfies the equation on either side of $t=1$. As in the previous case, energy estimates and an application of a suitable negative power of the operator $\langle D\rangle$, if required, show again that the generalised solution is associated with the classical distributional solution.

\appendix
\section{Appendix: G\r{a}rding's inequality for a matrix of generalised pseudodifferential operators}
In the sequel we make use of the following notations concerning matrices of symbols:
\begin{itemize}
\item[-] $A(x,\xi)^\ast=\overline{A(x,\xi)}^T$,
\item[-] $\Re\, A(x,\xi)=\frac{A(x,\xi)+A(x,\xi)^\ast}{2}$,
\item[-] $A(x,\xi)\ge 0$ if and only if $\overline{z}^TA(x,\xi)z\ge 0$ for all $z\in\C^m$ with $z\neq 0$,
\item[-] $A(x,\xi)\ge B(x,\xi)$ if and only if $(A-B)(x,\xi)\ge 0$.
\end{itemize}

The proof of the G\r{a}rding's inequality as stated in \cite[Theorem 4.4]{Kumano-go:81} requires the following preliminary results on the Friedrichs part of a pseudodifferential operator. For more details we refer to \cite[Chapter 3, Section 4]{Kumano-go:81}.
\begin{definition}
\label{def_Friedrichs}
Let $q(\sigma)\in\Cinfc(|\sigma|<1)$ be an even function with $\int q(\sigma)^2\, d\sigma=1$. Let
\[
F(\xi,\zeta)=q((\zeta-\xi)\lara{\xi}^{-1/2})\lara{\xi}^{-n/4}.
\]
Let $p\in S^m(\R^{2n})$. The Friedrichs part $p_F(\xi,x',\xi')$ of the symbol $p(x,\xi)$ is defined by
\[
p_F(\xi,x',\xi')=\int F(\xi,\zeta)p(x',\zeta)F(\xi',\zeta)\, d\zeta.
\]
\end{definition}
Note that if $p\in S^m(\R^{2n})$ then $p_F(\xi,x',\xi')\in S^{m,0}_{1/2,0}(\R^{3n})$, i.e.,
\[
|\partial^{\alpha}_\xi\partial^{\alpha'}_{\xi'}\partial^{\beta'}_{x} p_F(\xi,x',\xi')|\le c\lara{\xi}^{m-\frac{1}{2}|\alpha|}\lara{\xi'}^{-\frac{1}{2}|\alpha'|}
\]
for all $\alpha,\alpha',\beta'\in\N^n$.

\begin{theorem}
\label{theo_Friedrichs_1}
Let $(p_\eps)_\eps\in\mM_{S^m(\R^{2n})}$. Then
\begin{itemize}
\item[(i)] the corresponding net $(p_{F,\eps})_\eps$ belongs to $\mM_{S^{m,0}_{1/2,0}(\R^{3n})}$;
\item[(ii)] the operator $p_{F,\eps}(D_x,x',D_{x'})$ can be written as a pseudodifferential operator with symbol $(\sigma_{F,\eps})_\eps\in \mM_{S^m(\R^{2n})}$;
\item[(iii)] $(p_\eps-\sigma_{F,\eps})_\eps\in\mM_{S^{m-1}(\R^{2n})}$.
\end{itemize}
\end{theorem}
The proof of Theorem \ref{theo_Friedrichs_1} is obtained by combining Theorem 4.2 in \cite[Chapter 3]{Kumano-go:81} with Theorems 2.5 and 3.1 in
\cite[Chapter 2]{Kumano-go:81}. A careful inspection of the proofs shows that the expected moderateness estimates hold. From Theorem 4.3 in \cite{Kumano-go:81} we easily obtain the following statement.
\begin{theorem}
\label{theo_Friedrichs_2}
Let $P_\eps=P_\eps(x,D_x)$ be an $l\times l$ matrix of pseudodifferential operators with $(P_{\eps})_\eps\in \mM_{S^m(\R^{2n})}$. Let $P_{F,\eps}$ be the Friedrichs part of $P_\eps$. Then the operator $P_{F,\eps}-P_\eps$ is given by a matrix with symbol entries in $\mM_{S^{m-1}(\R^{2n})}$. In addition,
\begin{itemize}
\item[(i)] if $P_\eps(x,\xi)$ is Hermitian, then $P_{F,\eps}$ is self-adjoint, i.e.,
\[
(P_{F,\eps}u,v)=(u,P_{F,\eps} v)
\]
for all $u,v$ with components in $\S(\R^n)$;
\item[(ii)] if $P_\eps(x,\xi)\ge 0$ then the operator $P_{F,\eps}$ is positive, i.e.,
\[
(P_{F,\eps}u,u)\ge 0
\]
for all $u$ with components in $\S(\R^n)$;
\item[(iii)] if $P_\eps(x,\xi)$ is skew-Hermitian ($P_\eps^\ast=-P_\eps$) then
\[
(P_{F,\eps}u,v)=-(u,P_{F,\eps} v)
\]
for all $u,v$ as above.
\end{itemize}
\end{theorem}
We are now ready to prove the G\r{a}rding inequality. For the sake of simplicity we only consider $0$-order generalised pseudodifferential operators. The norms used in the following theorem are the Sobolev norms of order $0$ and $-1/2$ respectively and $(\cdot,\cdot)$ denotes the Hilbert product in ${L^2(\R^n)}^m$ (or equivalently ${H^0(\R^n)}^m$).
\begin{theorem}
\label{theo_Garding}
Let $A(x,D_x)$ be an $m\times m$ matrix of generalised pseudodifferential operators with symbol in $\G_{S^0(\R^{2n})}$. If there exist a representative $(A_\eps)_\eps$ of $A(x,\xi)$ and strictly nonzero nets $(c_\eps)_\eps$ and $(r_\eps)_\eps$ such that
\beq
\label{bound_re}
\Re\, A_\eps(x,\xi)\ge c_\eps I
\eeq
for $|\xi|\ge r_\eps$ then there exists a moderate net $(c_{1,\eps})_\eps$ such that
\beq
\label{Garding}
\Re(A_\eps(x,D_x)u,u)\ge c_{\eps}\Vert u\Vert^2_0-c_{1,\eps}\Vert u\Vert^2_{-1/2}
\eeq
for all vectors $u$ with components in $\S(\R^n)$ and all $\eps\in(0,1]$.
\end{theorem}
\begin{proof}
Let $A_{1,\eps}(x,\xi)=A_{\eps}(x,\xi)-c_\eps I$. Then
\[
\Re(A_{1,\eps}(x,D_x)u,u)=\Re(A_\eps(x,D_x)u,u)-c_\eps\Vert u\Vert^2_0.
\]
We define the operators $B_{\eps}(x,D_x)$ and $C_{\eps}(x,D_x)$ via the symbols $\Re\, A_\eps(x,\xi)-c_\eps I$ and $(A_\eps(x,\xi)-A_\eps(x,\xi)^\ast)/2$, respectively. Hence $A_{1,\eps}(x,D_x)=B_\eps(x,D_x)+C_\eps(x,D_x)$. Making use of the hypothesis \eqref{bound_re} and multiplying the entries of $B_\eps$ by a suitable cut-off net $(\psi_\eps(\xi))_\eps=(\psi(\xi/r_\eps))_\eps$ we obtain that there exist matrices $D_\eps(x,\xi)$ and $R_\eps(x,\xi)$ with entries in $\mM_{S^0(\R^{2n})}$ and $\mM_{S^{-\infty}(\R^{2n})}$ respectively such that
\[
A_{1,\eps}(x,\xi)=D_\eps(x,\xi)+C_\eps(x,\xi)+R_\eps(x,\xi),
\]
and $D_\eps(x,\xi)\ge 0$. We now take the Friedrichs part of the operators $D_\eps(x,D_x)$ and $C_\eps(x,D_x)$. Theorem \ref{theo_Friedrichs_2} yields
\[
A_{1,\eps}(x,D_x)=D_{F,\eps}(x,D_x)+C_{F,\eps}(x,D_x)+S_\eps(x,D_x),
\]
where $S_\eps$ has symbol entries in $\mM_{S^{-1}(\R^{2n})}$, $D_{F,\eps}(x,D_x)$ is positive and $(C_\eps(x,D_x)u,v)=-(u, C_\eps(x,D_x)v)$ for all $u,v$ with components in $\S(\R^n)$. Hence, combining this decomposition with the Sobolev mapping properties of generalised pseudodifferential operators we get the inequality
\begin{multline}
\label{est_re_1}
\Re(A_{1,\eps}(x,D_x)u,u)=(D_{F,\eps}(x,D_x)u,u)+\Re(S_\eps(x,D_x)u,u)\ge -\Vert S_\eps(x,D_x)u\Vert_{1/2}\Vert u\Vert_{-1/2}\\
\ge -c_{1,\eps}\Vert u\Vert_{-1/2}^2,
\end{multline}
where the net $(c_{1,\eps})_\eps$ depends on the symbol seminorms of the entries of $(S_\eps(x,\xi))_\eps$. The assertion \eqref{Garding} follows immediately from \eqref{est_re_1}.
\end{proof}
We conclude this appendix by stating the G\r{a}rding inequality under different assumptions of moderateness at the level of representatives. Focusing on the Friedrichs part of a pseudodifferential operator one can state Theorem \ref{theo_Friedrichs_1} in terms of slow scale symbols and logarithmic slow scale symbols. These two new versions, more specific under the point of view of the moderateness assumptions, are stated by replacing $\mM$ with $\mM^{\ssc}$ and $\mM^{\rm{lsc}}$ respectively, throughout the statement of the theorem. This leads to the following G\r{a}rding inequalities.
\begin{theorem}
\label{theo_Garding_23}
\leavevmode
\begin{itemize}
\item[(i)] Let $A(x,D_x)$ be an $m\times m$ matrix of generalised pseudodifferential operators with symbol in $\G^\ssc_{S^0(\R^{2n})}$. If there exist a representative $(A_\eps)_\eps$ of $A(x,\xi)$ and slow scale nets $(c^{-1}_\eps)_\eps$ and $(r^{-1}_\eps)_\eps$ such that
\beq
\label{bound_re_sc}
\Re\, A_\eps(x,\xi)\ge c_\eps I
\eeq
for $|\xi|\ge r_\eps$ then there exists a slow scale net $(c_{1,\eps})_\eps$ such that
\beq
\label{Garding_sc}
\Re(A_\eps(x,D_x)u,u)\ge c_{\eps}\Vert u\Vert^2_0-c_{1,\eps}\Vert u\Vert^2_{-1/2}
\eeq
for all vectors $u$ with components in $\S(\R^n)$ and all $\eps\in(0,1]$.
\item[(ii)] Let $A(x,D_x)$ be an $m\times m$ matrix of generalised pseudodifferential operators with symbol in $\G^{\rm{lsc}}_{S^0(\R^{2n})}$. If there exist a representative $(A_\eps)_\eps$ of $A(x,\xi)$ and logarithmic slow scale nets $(c^{-1}_\eps)_\eps$ and $(r^{-1}_\eps)_\eps$ such that
\beq
\label{bound_re_lsc}
\Re\, A_\eps(x,\xi)\ge c_\eps I
\eeq
for $|\xi|\ge r_\eps$ then there exists a logarithmic slow scale net $(c_{1,\eps})_\eps$ such that
\beq
\label{Garding_lsc}
\Re(A_\eps(x,D_x)u,u)\ge c_{\eps}\Vert u\Vert^2_0-c_{1,\eps}\Vert u\Vert^2_{-1/2}
\eeq
for all vectors $u$ with components in $\S(\R^n)$ and all $\eps\in(0,1]$.
\end{itemize}
\end{theorem}

\bibliographystyle{abbrv}
\newcommand{\SortNoop}[1]{}

\end{document}